\documentclass[11pt]{amsart}
\usepackage{float}
\usepackage{etex}
\usepackage[margin={1.2in}]{geometry}
\usepackage[titletoc]{appendix}
\usepackage{amsmath,amscd}
\usepackage{graphicx}

\usepackage[colorlinks=true]{hyperref}
\usepackage{url}

\usepackage{enumerate}

\usepackage{mathtools}
\usepackage{tikz}
\usetikzlibrary{matrix}
\usepackage{mathrsfs}
\usepackage{stmaryrd}

\usepackage{rotating}

\usepackage{longtable}

\usepackage[sc]{mathpazo}
\usepackage{courier}

\usepackage[utf8]{inputenc}
\usepackage[T1]{fontenc}

\usepackage[english]{babel}
\usepackage{blindtext}

\usepackage{amsmath}

\usepackage{microtype}

\usepackage{euscript}

\usepackage{amssymb}
\usepackage{amsthm}
\usepackage{color}
\usepackage{latexsym,extarrows}

\usepackage[all]{xy}

\usepackage[stable,multiple]{footmisc}

\makeatletter
\newsavebox{\@brx}
\newcommand{\llangle}[1][]{\savebox{\@brx}{\(\m@th{#1\langle}\)}%
  \mathopen{\copy\@brx\kern-0.5\wd\@brx\usebox{\@brx}}}
\newcommand{\rrangle}[1][]{\savebox{\@brx}{\(\m@th{#1\rangle}\)}%
  \mathclose{\copy\@brx\kern-0.5\wd\@brx\usebox{\@brx}}}
\makeatother

\begin{document}
\def\e#1\e{\begin{equation}#1\end{equation}}
\def\ea#1\ea{\begin{align}#1\end{align}}
\def\eq#1{{\rm(\ref{#1})}}
\theoremstyle{plain}
\newtheorem{thm}{Theorem}
\newtheorem{lem}{Lemma}
\newtheorem{prop}{Proposition}
\newtheorem{cor}{Corollary}
\newtheorem{rem}{Remark}
\newtheorem{ex}{Example}
\newtheorem{conj}{Conjecture}

\newcommand{\D}{\mathcal{D}}
\newcommand{\A}{\mathcal{A}}
\newcommand{\LL}{\llangle[\big]}
\newcommand{\RR}{\rrangle[\big]}
\newcommand{\LD}{\big\langle}
\newcommand{\RD}{\big\rangle}
\newcommand{\C}{\mathbb{C}}
\newcommand{\F}{\mathcal{F}}
\newcommand{\HH}{\mathcal{H}}
\newcommand{\X}{\mathcal{X}}
\newcommand{\Z}{\mathbb{Z}}
\newcommand{\E}{\mathcal{E}}
\newcommand{\M}{\mathcal{M}}
\newcommand{\W}{\mathcal{W}}
\newcommand{\PP}{\mathbb{P}}
\newcommand{\K}{\mathscr{K}}
\newcommand{\Q}{\mathfrak{Q}}
\newcommand{\q}{\mathbf{q}}
\newcommand{\one}{\mathbf{1}}
\newcommand{\wpsi}{\widetilde{\psi}}
\newcommand{\bpsi}{\bar{\psi}}

\def\sub{\subset} \def\lra{\longrightarrow}

\numberwithin{equation}{section}

\def\red{\color{red}}

\newcommand{\comment}[1]{\textcolor{red}{[#1]}} 


\def\fl{{\mathrm{fl}}}

\def\umv{^{\mathrm{mv}}}
\def\lophi{_{^{(1,\varphi)}}}

\def\barH{{\bar H}}
\def\lorho{_{^{(1,\rho)}}}

\def\barM{{\bcM}}
\def\barG{{\overline{G}}}
\def\barS{{\overline{S}}}
\def\fF{\mathfrak{F}}

\def\fr{\mathfrak{r}}
\def\TF{T\fF}
\def\lgd{_{g,d}}
\def\fdgp{{\fD_g^p}}
\def\fdg{{\fD_g}}
\def\tfdg{{\widetilde\fD_g}}
\def\fdmp{{\fD^p_\fM}}
\def\fd{{\fD}}
\def\tfd{{\ti\fD}}
\def\fdp{{\fD^p}}
\def\tfdp{{\ti\fD^p}}
\def\fdm{{\fD_\fM}}
\def\tfdo{{\widetilde\fD_1}}
\def\defeq{:=}\def\bB{{\mathbf B}}
\def\bA{{\mathbf A}}
\def\bC{{\mathbf C}}
\def\bXc{{\bX^\circ}}
\def\ff{{\mathfrak f}}
\def\ffsta{\ff\sta}
\def\tell{\text{ell}}
\def\tgst{\text{gst}}
\def\bcM{\overline{M}}

\def\Sbul{\mathrm{Sym}}
\def\be{{\mathbf e}}
\numberwithin{equation}{section}
\def\sP{{\mathscr P}}

\newcommand{\Spec}{\operatorname{Spec}}
\newcommand{\Proj}{\operatorname{Proj}}

\def\lle{_{\tcX_\mu,\mathrm{pr}}}
\def\llr{_{\tcX_\mu,\mathrm{tl}}}
\def\llg{_{\mu,\mathrm{gst}}}
\def\fh{{\mathfrak h}}
\def\bv{{\mathbf v}}
\def\sP{{\mathscr P}}
\def\Moo{\bcM_{1,1}}
\def\sP{{\mathscr P}}
\def\sQ{{\mathscr Q}}

\def\fY{{\mathfrak Y}}
\def\fX{{\mathfrak X}}
\def\bQ{{\mathbf Q}}
\def\lzc{_0^\circ}
\def\lzo{_{o\cap \mu}}
\def\bg{{\mathfrak g}}
\def\oWc{{\overline W\dcirc}}
\def\sY{{\mathscr Y}}
\def\sC{{\mathscr C}}
\def\sO{{\mathscr O}}
\def\sM{{\mathscr M}}
\def\sN{{\mathscr N}}
\def\sL{{\mathscr L}}
\def\sr{{\mathscr r}}
\def\sI{{\mathscr I}}
\def\sO{\mathscr{O}}
\def\sI{\mathscr{I}}
\def\sJ{\mathscr{J}}
\def\sE{\mathscr{E}}
\def\sF{\mathscr{F}}
\def\sV{\mathscr{V}}
\def\sA{\mathscr{A}}
\def\sU{\mathscr{U}}
\def\sR{\mathscr{R}}
\def\sX{\mathscr{X}}
\newcommand{\CC}{\mathbb{C}}
\newcommand{\EE}{\mathbb{E}}
\newcommand{\QQ}{\mathbb{Q}}
\newcommand{\ZZ}{\mathbb{Z}}
\newcommand{\NN}{\mathbb{N}}
\newcommand{\GG}{\mathbb{G}}
\newcommand{\FF}{\mathbb{F}}
\newcommand{\fP}{\mathbb{M}}
\newcommand{\UU}{\mathbb{U}}
\newcommand{\bUU}{\overline{\mathbb{U}}}
\newcommand{\VV}{\mathbb{V}}
\newcommand{\WW}{\mathbb{W}}
\newcommand{\TT}{\mathbb{T}}
\def\bfV{\bar\fV}

\def\shar{^!}
\def\intr{^{int}}
\def\sK{{\mathscr K}}
\def\ufour{^{\oplus 4}}
\def\subs{\subsection{}}

\def\fh{{\mathfrak h}}
\def\fn{{\mathfrak n}}
\def\fa{{\mathfrak a}}
\def\fb{{\mathfrak b}}
\def\fR{{\mathfrak R}}
\def\psta{^{\prime \ast}}
\def\vd{{\vec{d}}}
\def\vm{{\vec{m}}}

\def\fq{{\mathfrak q}}
\def\sP{{\mathscr P}}
\def\lred{_{\text{red}}}
\def\prista{^{\prime\ast}}
\def\llamp{_{\lamp}}
\def\lDel{_{\Delta}}
\def\fn{{\mathfrak n}}
\def\fk{{\mathfrak k}}
\def\HH{\mathbb H}
\def\fK{{\mathfrak K}}
\def\fH{{\mathfrak H}}

\def\la{\big\langle}
\def\ra{\big\rangle}

\def\fQ{{\mathfrak Q}}
\def\fL{{\mathfrak L}}
\def\Total{\mathrm{Total}}
\def\cqg{{\cQ_{g}}}
\def\nq{{N}}
\def\cng{{\cN_{g}}}
\def\cngp{\cN} 
\def\Vb{\mathrm{Vb}}
\def\cpd{\cP}
\def\cpg{{\cP_g}}
\def\cpgp{\cP_g}
\def\cvg{{\cV_g}}
\def\cvgp{\cV} 
\def\nn{\mathfrak n}
\def\fS{\mathfrak S}
\def\fy{\mathfrak y}
\def\vdim{\mathrm{vir}.\dim}

\def\cno{{\cN_{1}}}
\def\cnop{\cN} 
\def\Vb{\mathrm{Vb}}
\def\cpd{\cP}
\def\cpxp{\cP_1}
\def\cpy{{\cP_\cY}}
\def\cpx{{\cP_\cX}}
\def\cvo{{\cV_1}}
\def\cvop{\cV} 

\def\ured{^{\text{red}}}

\def\boldY{{\mathbf Y}}
\def\boldX{{\mathbf X}}
\def\boldE{{\mathbf E}}
\def\boldF{{\mathbf F}}
\def\fk{{\mathfrak k}}
\def\ufl{^{\text{flat}}}
\def\ellip{\text{ell}}
\def\gst{\text{gst}}
\def\lred{_{\text{red}}}
\def\cA{{\mathcal A}}

\def\fdg{{\fD_g}}
\def\tfdg{{\widetilde\fD_g}}
\def\AA{\mathbb A}
\def\lAo{_{\Ao}}


\newcommand{\bL}{\mathbf{L}}
\newcommand{\bT}{\mathbf{T}}
\def\bS{\mathbf{S}}
\newcommand{\bk}{\mathbf{k}}
\newcommand{\bfl}{\mathbf{l}}
\newcommand{\bm}{\mathbf{m}}
\newcommand{\bn}{\mathbf{n}}
\newcommand{\bp}{\mathbf{p}}
\newcommand{\bq}{\mathbf{q}}
\newcommand{\bs}{\mathbf{s}}
\newcommand{\bt}{\mathbf{t}}
\newcommand{\bw}{\mathbf{w}}
\newcommand{\bx}{\mathbf{x}}
\def\bE{{\mathbf E}}
\def\bF{{\mathbf F}}
\def\bG{{\mathbf G}}
\def\bK{{\mathbf K}}
\def\bH{{\mathbf H}}
\def\bR{{\mathbf R}}
\def\tV{{\texttt{V}}}
\newcommand{\kk}{\bk}
\def\bA{\mathbf{A}}
\newcommand{\bP}{\mathbf{P}}

\def\lggd{_{g,\gamma,\bd}}
\def\cWgg{\cW\lggd}


\newcommand{\cal}{\mathcal}
\def\oR{{\overline{R}}}
\def\cA{{\cal A}}
\def\cB{{\cal B}}
\def\cC{{\cal C}}
\def\cD{{\cal D}}
\def\cE{{\cal E}}
\def\cF{{\cal F}}
\def\cH{{\cal H}}
\def\cK{{\cal K}}
\def\cL{{\cal L}}
\def\cM{{\cal M}}
\def\cN{{\cal N}}
\def\cO{{\cal O}}
\def\cP{{\cal P}}
\def\cQ{{\cal Q}}
\def\cR{{\cal R}}
\def\cT{{\cal T}}
\def\cU{{\cal U}}
\def\cV{{\cal V}}
\def\cW{{\cal W}}
\def\cS{{\cal S}}
\def\cX{{\cal X}}
\def\cY{{\cal Y}}
\def\cZ{{\cal Z}}
\def\cI{{\cal I}}

\def\tint{\text{int}}
\newcommand{\oS}{\mathring{S}}


\def\fB{\mathfrak{B}}
\def\bfB{\bar\fB}
\def\fC{\mathfrak{C}}
\def\fD{\mathfrak{D}}
\def\fE{\mathfrak{E}}
\def\fF{\mathfrak{F}}
\def\fP{\mathfrak{M}}
\def\fV{\mathfrak{V}}
\def\fX{\mathfrak{X}}
\def\fy{\mathfrak{Y}}

\def\ff{\mathfrak{f}}
\def\fm{\mathfrak{m}}
\def\fp{\mathfrak{p}}
\def\ft{\mathbf{q}}
\def\fu{\mathfrac{u}}
\def\fv{\mathfrak{v}}
\def\fe{\mathfrak{e}}
\def\fs{\mathfrak{s}}
\def\frev{\mathfrak{rev}}
\def\tcM{\tilde\fM}


\newcommand{\tPhi }{\tilde{\Phi} }
\newcommand{\tGa}{\tilde{\Gamma}}
\newcommand{\tbeta}{\tilde{\beta}}
\newcommand{\trho }{\tilde{\rho} }
\newcommand{\tpi  }{\tilde{\pi}  }


\newcommand{\tC}{\tilde{C}}
\newcommand{\tD}{\tilde{D}}
\newcommand{\tE}{\tilde{E}}
\newcommand{\tF}{\tilde{F}}
\newcommand{\tK}{\tilde{K}}
\newcommand{\tL}{\tilde{L}}
\newcommand{\tT}{\tilde{T}}
\newcommand{\tY}{\tilde{Y}}

\newcommand{\tf}{\tilde{f}}
\newcommand{\tih}{\tilde{h}}
\newcommand{\tz}{\tilde{z}}
\newcommand{\tu}{\tilde{u}}
\newcommand{\tpsi}{\widetilde{\psi}}
\newcommand{\tkappa}{\tilde{\kappa}}

\def\tilcW{{\tilde\cW}}
\def\tilcM{\tilde\cM}
\def\tilcC{{\tilde\cC}}
\def\tilcZ{\tilde\cZ}
\def\tilcX{\tilde\cX}
\def\tilcV{{\tilde\cV}}
\def\tilcW{\tilde\cW}
\def\tilcD{{\tilde\cD}}
\def\tilY{{\tilde Y}}

\def\tilX{{\tilde X}}
\def\tilV{{\tilde V}}
\def\tilW{{\tilde W}}
\def\tilZ{{\tilde Z}}
\def\tilq{{\tilde q}}
\def\tilp{{\tilde q}}


\newcommand{\hD}{\widehat{D}}
\newcommand{\hL}{\widehat{L}}
\newcommand{\hT}{\widehat{T}}
\newcommand{\hX}{\widehat{X}}
\newcommand{\hY}{\widehat{Y}}
\newcommand{\hZ}{\widehat{Z}}
\newcommand{\hf}{\widehat{f}}
\newcommand{\hu}{\widehat{u}}
\newcommand{\hDe}{\widehat{Delta}}


\newcommand{\eGa}{\check{\Ga}}


\newcommand{\Mbar}{\overline{\cM}}
\newcommand{\MX}{\Mbar^\bu_\chi(\Gamma,\vec{d},\vmu)}
\newcommand{\GX}{G^\bu_\chi(\Gamma,\vec{d},\vmu)}
\newcommand{\Mi}{\Mbar^\bu_{\chi^i}(\Po,\nu^i,\mu^i)}
\newcommand{\Mv}{\Mbar^\bu_{\chi^v}(\Po,\nu^v,\mu^v)}
\newcommand{\GYfm}{G^\bu_{\chi,\vmu}(\Gamma)}
\newcommand{\Gdmu}{G^\bu_\chi(\Gamma,\vd,\vmu)}
\newcommand{\Mdmu}{\cM^\bu_\chi(\widehatYrel,\vd,\vmu)}
\newcommand{\tMd}{\tilde{\cM}^\bu_\chi(\widehatYrel,\vd,\vmu)}

\def\mzd{\cM^\bu_{\chi}(Z,d)}
\def\mzod{\cM^\bu_{\chi}(Z^1,d)}
\def\mhzd{\cM^\bu_{\chi}(\widehat Z,d)}
\def\Mwml{\cM^\bu_{\chi,\vd,\vmu}(W\urel,L)}
\def\Mfml{\cM_{\chi,\vd,\vmu}^\bu(\widehatYrel,\hL)}
\def\Mgw{\cM_{g,\vmu}(W\urel,L)}
\def\mxdm{\cM_{\chi,\vd,\vmu}}
\def\Mfmlz{\mxdm(\hY_{\Gamma^0}\urel, \hL)}
\def\Mgz{\cM_{g,\vmu}(\hY_{\Ga^0}\urel,\hL)}
\def\Mmu{\cM^\bu_{\chi,\vmu}(\Gamma)}
\def\mapright#1{\,\smash{\mathop{\lra}\limits^{#1}}\,}
\def\maprightsmall#1{\,\smash{\mathop{\lra}\limits^{{}_{#1}}}\,}
\def\twomapright#1{\,\smash{\mathop{-\!\!\!\lra}\limits^{#1}}\,}
\def\mapdown#1{\ \downarrow\!{#1}}
\def\Myy{\cM_{\chi,\vd,\vmu}^\bu(Y_{\Gamma^0}\urel,L)}
\def\mw{\cM_{\chi,\vd,\vmu}^\bu(W\urel,L)}

\def\my{\cM^\bu(\Gamma,\widehat L)}
\def\myy{\cM^\bu(Y\urel,L)}
\def\mwtdef{\cM^\soe\ldef}
\def\MY{\cM^\bu_\chi(\widehat{\cY},\vd,\vmu)}
\let\cMM=\cM
\let\cNN=\cN

\def\llam{_{\lambda}}
\def\lsi{_{\sigma}}

\def\BM{H^{BM}\lsta}


\newcommand{\two  }[1]{ ({#1}_1,{#1}_2) }
\newcommand{\three}[1]{ ({#1}_1,{#1}_2,{#1}_3) }

\def\oL{{\overline{L}}}

\newcommand{\pair}{(\vx,\vnu)}
\newcommand{\vmu}{{\vec{\mu}}}
\newcommand{\vnu}{{\vec{\nu}}}
\newcommand{\vx}{\vec{\chi}}
\newcommand{\vz}{\vec{z}}
\newcommand{\vsi}{\vec{\sigma}}
\newcommand{\vn}{\vec{n}}
\newcommand{\Id}{{\rm Id}}
\newcommand{\wt}{\widetilde{t}}
\newcommand{\ws}{\widetilde{s}}

\newcommand{\up}[1]{ {{#1}^1,{#1}^2,{#1}^3} }

\def\dual{^{\vee}}
\def\ucirc{^\circ}
\def\sta{^\ast}
\def\st{^{\mathrm{st}}}
\def\virt{^{\mathrm{vir}}}
\def\upmo{^{-1}}
\def\sta{^{\ast}}
\def\dpri{^{\prime\prime}}
\def\ori{^{\mathrm{o}}}
\def\urel{^{\mathrm{rel}}}
\def\pri{^{\prime}}
\def\virtt{^{\text{vir},\te}}
\def\virts{^{\text{vir},\soe}}
\def\virz{^{\text{vir},T}}
\def\mm{{\mathfrak m}}
\def\sta{^*}
\def\uso{^{S^1}}
\def\oV{\overline{V}}
\def\sB{{\mathscr B}}


\newcommand{\lo}[1]{ {{#1}_1,{#1}_2,{#1}_3} }
\newcommand{\hxn}[1]{ {#1}_{\widehat{\chi},\widehat{\nu}} }
\newcommand{\xmu}[1]{{#1}^\bu_{\chi,\up{\mu}} }
\newcommand{\xnu}[1]{{#1}^\bu_{\chi,\up{\nu}} }
\newcommand{\gmu}[1]{{#1}_{g,\up{\mu}}}
\newcommand{\xmm}[1]{{#1}^\bu_{\chi,\mm}}
\newcommand{\xn}[1]{{#1}_{\vx,\vnu}}

\def\lbe{_{\beta}}
\def\lra{\longrightarrow}
\def\lpri{_{\mathrm{pri}}}
\def\lsta{_{\ast}}
\def\lvec{\overrightarrow}
\def\Oplus{\mathop{\oplus}}
\def\las{_{a\ast}}
\def\ldef{_{\text{def}}}


\newcommand{\Del}{\Delta}
\newcommand{\Si}{\Sigma}
\newcommand{\Ga}{\Gamma}

\newcommand{\ep}{\epsilon}
\newcommand{\lam}{\lambda}
\newcommand{\si}{\sigma}


\newcommand{\xnm}{ {-\chi^i \Delta\ell(\nu^i)\Delta \ell(\mu^i)}}
\newcommand{\vnm}{ {-\chi^v \Delta\ell(\nu^v)\Delta \ell(\mu^v)}}
\newcommand{\lmu}{_\mu}
\newcommand{\pa}{\partial}
\newcommand{\bu}{\bullet}
\newcommand{\da}{D^\alpha }
\newcommand{\Da}{\Delta(D^\alpha)}
\newcommand{\dpm}{\Delta^\pm}
\newcommand{\Dp}{\Delta^\Delta}
\newcommand{\Dm}{\Delta^-}
\newcommand{\TY}{u^*\left(\Omega_{Y_\bm}(\log \hD_\bm)\right)^\vee }
\newcommand{\Ym}{Y[m^1,\ldots,m^k]}
\newcommand{\pim}{\pi[m^1,\ldots,m^k]}
\newcommand{\pmu}{p^1_{\mu^1}p^2_{\mu^2}p^3_{\mu^3}}
\newcommand{\TZ}{u^*\left(\Omega_Z(\sum_{\alpha=1}^k
                 \log D^\alpha_{(m^\alpha)})\right)^\vee }
\newcommand{\ee}{{\bar{e}}}
\newcommand{\bfmS}{{\bf mS}}

\def\boldb{{\mathbf b}}

\def\ocZ{{\overline{\cZ}}}
\def\begeq{\begin{equation}}
\def\endeq{\end{equation}}
\def\and{\quad{\rm and}\quad}
\def\bl{\bigl(}
\def\br{\bigr)}
\def\defeq{:=}
\def\mh{\!:\!}
\def\sub{\subset}
\def\Ao{{\mathbb A}^{\!1}}
\def\Aosta{{\mathbb A}^{\!1\ast}}
\def\Zt{\ZZ^{\oplus 2}}
\def\widehatYrel{\widehat Y_\Gamma\urel}
\def\Pt{{\mathbb P}^2}
\def\Po{{\mathbb P^1}}
\def\and{\quad\text{and}\quad}
\def\mapleft#1{\,\smash{\mathop{\longleftarrow}\limits^{#1}}\,}
\def\lmapright#1{\,\smash{\mathop{-\!\!\!\lra}\limits^{#1}}\,}
\def\llra{\mathop{-\!\!\!\lra}}
\def\lt{{\lambda\lzo}_T^\vee}
\def\ltp{{\lambda\lzo}^\vee_{T,\mathrm{pri} }}
\def\lalp{_\alpha}
\let\mwnew=\mw
\def\pkt{\phi_{k,t}}
\def\wm{W[\bm]}
\def\dm{D[\bm]}
\def\so{{S^1}}
\def\soe{{T_\eta}}
\def\ob{\text{ob}}
\def\te{T}
\def\tz{T}
\def\tet{{T_\eta}}
\def\pkci{\phi_{k_i,c_i}}
\def\bbl{\Bigl(}
\def\bbr{\Bigr)}
\def\reg{_{\text{reg}}}

\def\llra{\,\mathop{-\!\!\!\lra}\,}

\def\dbar{\overline{\partial}}
\def\bone{{\mathbf 1}}
\def\mgn{\cM_{g,n}}
\def\mxnxd{\cM^\bullet_{\chi,d,n}(\cX/T)}
\def\mxnxno{\cM^\bullet_{\chi_1,\nu_1,n_1}(X_1/E_1)}
\def\mxnxn{\cM^\bullet_{\chi,\nu,n}(X/E)}
\def\mxnd{\cM^\bullet_{\chi,d,n}}
\def\mxn{\cM_{\chi,n}}

\def\Ob{\cO b}
\def\An{{{\mathbb A}^n}}
\def\kzz{\kk[\![z]\!]}

\def\lep{_{\ep}}
\def\lloc{_{\mathrm{loc}}}
\def\uep{^\ep}
\def\loc{{\mathrm{loc}}}
\def\lpbar{_{\bar p}}
\def\mcn{\cM_{\chi,n}}

\def\lalpbe{_{\alpha\beta}}
\def\bul{^\bullet}

\def\ev{\mathrm{ev}}
\def\uso{^{S^1}}

\def\aa{\alpha}
\def\tcX{\tilde{\cX}}
\def\tcY{\tilde{\cY}}
\def\tcD{\tilde{\cD}}

\def\bnu{\bar{\nu}}
\def\bth{\bar{\theta}}
\def\bp{\bar{p}}

\def\mxdb{\cM_{\cX,d}^\bullet}
\def\myib{\cM_{Y_i,\eta_i}^\bullet}
\def\myb{\cM_{Y,\eta}^\bullet}
\def\myy{\cM_{Y_1\cup Y_2,\eta}\bul}
\def\bsig{{\bar{\sigma}}}

\def\Pn{{\mathbb P}^n}
\def\MPdd{\cM_2(\Pn,d)}
\def\MPd{\cM_1(\Pn,d)}
\def\MPod{\cM_{0,1}(\Pn,d)}
\def\Ng{\cN_\eta}
\def\sta{^\ast}
\def\image{\text{Im}\,}
\def\lineN{\overline{\cN}}
\let\lab=\label
\def\lone{_{o[a]}}

\def\lineM{\overline{\cM}}
\def\etao{{\eta_1}}
\def\etat{{\eta_2}}
\def\rhat{\widehat{}}

\def\tilV{{\tilde V}}
\def\tilE{{\tilde E}}
\def\tilS{{\tilde S}}

\def\llam{_{\lambda}}
\def\lpr{_{\text{pr}}}

\def\oev{{\overline{ev}}}

\def\bQ{{\mathbf Q}}
\def\lzc{_0^\circ}
\def\lpo{_1^\circ}

\def\ocY{{\overline{\cY}}}
\def\sO{{\mathscr O}}
\def\sW{{\mathscr W}}
\def\sG{{\mathscr G}}
\def\sH{{\mathscr H}}
\def\sR{{\mathscr R}}
\def\sD{{\mathscr D}}

\def\beq{\begin{equation}}
\def\eeq{\end{equation}}
\def\backl{{/}}
\def\ppri{^{\prime\prime}}
\def\uplusn{^{\oplus n}}
\def\vsp{\vskip5pt}
\def\Pf{{\PP^4}}

\def\shar{^!}
\def\intr{^{int}}
\def\bM{\mathbf{M}}

\def\bbP{\overline{\bP}_{g,d}}
\def\bbM{\overline{\fM}}
\def\bbPef{\bbP^{{}{\text{ef}}}}
\def\bbPo{\overline{\bP}_{1,d}^{\text{ef}}}
\def\bD{{\mathbf D}}
\def\bcMo{\overline{\cM}_1(\Pf,d)}
\def\bcMQo{\overline{\cM}_1(Q,d)}
\def\cMo{{\overline\cM}_1(\Pf,d)}
\def\cMQo{{\overline\cM}_1(Q,d)}
\def\fP{{\mathfrak P}}
\def\fM{{\mathfrak M}}
\def\uf{^{\oplus 4}}
\def\ut{^{\oplus 3}}

\def\fU{{\mathfrak U}}
\def\lod{_{1,(d)}}
\def\tilpsi{{\tilde\psi}}
\def\fZ{{\mathfrak Z}}
\def\of{^{\otimes 5}}
\def\bee{\begin{equation}}
\def\eeq{\end{equation}}

\def\bbMup{\bbM\lod^{bp}}

\def\fA{{\mathfrak A}}
\def\bV{{\mathbf V}}
\def\fW{{\mathfrak W}}
\def\uo{^\circ}
\def\ufive{^{\oplus 5}}

\def\bMgw{\overline{\cM}_{g,(d)}}
\def\tilM{{\tilde M}}
\def\uwt{^{{\textnormal d}}}

\def\fN{{\mathfrak N}}
\def\boldr{{\mathbf r}}
\def\bd{{\mathbf d}}
\def\lred{_{\text{red}}}

\def\tfm{{\tilde\fM\uwt}}
\def\ufive{^{\oplus 5}}
\def\ofour{^{\otimes 4}}
\def\ofive{^{\otimes 5}}
\def\eset{\emptyset}
\let\eps=\epsilon
\let\veps=\varepsilon
\def\tilPhi{{\Phi}}
\def\ust{^{\textnormal{st}}}
\def\ti{\tilde}
\def\tU{{\ti U}}
\def\Lam{{{\lambda\lzo}}}
\def\tS{{\ti S}}
\def\tT{{\ti T}}
\def\fg{{\mathfrak g}}

\def\lamp{{\lam'}}
\def\ufdpo{^{\oplus 5d\Delta1}}

\def\OB{\mathbf{Ob}}
\def\can{\cong_{\textnormal{can}}}
\def\Inc{{\sub}_1^4}
\def\Lfs{q\sta p\lsta \sL\ofive(\cS)}
\def\txi{{\ti\xi}}
\def\lD{_{\Delta}}
\def\lgst{_{\rm gst}}
\def\lell{_{\rm pri}}
\def\lzo{_{\rm int}}
\def\gst{{\rm{gst}}}
\def\el{\text{ell}}
\def\tiY{{\ti Y}}
\def\lcan{_{{\mathrm{can}}}}
\def\Gm{{G_{\mathrm m}}}

\def\cfM{{\cM^{\mathrm w}}}
\let\cfD=\cD
\def\tfd{{\tilde \cD}}
\def\cfN{{{\tilde\cM}^{\mathrm w}}}
\def\cpx{{\tilde \cY}}
\def\ticfM{\ti\cM^{\text{w}}}
\def\tcX{{\ti\cX}}
\def\tcy{{\ti\cY}}
\def\bcpri{\bC_{\mathrm{pri}}}
\def\bcgst{\bC\lgst}
\def\tcygst{{\tcy\lgst}}
\def\tfg{{T\fF_{g,d}}}
\def\tfo{{T\fF_{1,d}}}
\def\tf0{{T\fF_{0,d}}}
\def\P5{{\PP^5}}
\def\lab#1{\label{#1}[{#1}]\  }
\let\lab=\label
\def\mr{{\mathring}}
\def\upm{^{{\mathrm m}}}
\def\dzi{\partial_{z_i}}
\def\dzj{\partial_{z_j}}
\def\dwj{\partial_{w_j}}
\def\ufix{^{{\mathbf f}}}
\def\umove{^{{\mathbf m}}}
\def\sZ{\mathscr Z}
 \def\veck{{\vec k}}
\def\ga{{\Gamma}}
\def\romann{{\mathrm{n}}}
\def\romanv{{\mathrm{v}}}
\def\top{{\mathrm{top}}}
\def\lrest{_{\mathrm{rest}}}
\def\barcM{\overline{\cM}}
\def\cJ{\mathcal J}

\newcommand{\tZ}{\widetilde{Z}}

\renewcommand{\arraystretch}{1.5}

\title{\bf Higher Genus FJRW Invariants of a Fermat Cubic}
\author{Jun Li, Yefeng Shen and Jie Zhou}
\date{}
\maketitle

\begin{abstract}

We  reconstruct all-genus Fan-Jarvis-Ruan-Witten invariants of a Fermat cubic Landau-Ginzburg space 
$(x_1^3+x_2^3+x_3^3: [\mathbb{C}^3/ \mu_3]\to\mathbb{C})$ from genus-one primary  invariants, using tautological relations and axioms of Cohomological Field Theories.
 The genus-one primary invariants  satisfy a Chazy equation by the Belorousski-Pandharipande relation.
 They are completely determined by a single genus-one invariant, which can be obtained from cosection localization and intersection theory on moduli of three spin curves.

We solve an all-genus {\em Landau-Ginzburg/Calabi-Yau Correspondence Conjecture} for the Fermat cubic Landau-Ginzburg space using Cayley transformation on quasi-modular forms.
This transformation relates two non-semisimple CohFT theories: 
the Fan-Jarvis-Ruan-Witten theory of the Fermat cubic polynomial 
and the Gromov-Witten theory of the Fermat cubic curve.  
As a consequence, Fan-Jarvis-Ruan-Witten invariants at any genus can be computed using Gromov-Witten invariants of the elliptic curve.
They also satisfy nice structures including holomorphic anomaly equations and Virasoro constraints.



\end{abstract}

{
\hypersetup{linkcolor=black}
\setcounter{tocdepth}{2} \tableofcontents
}

\section{Introduction}
\label{secintro}

Let $(d; \delta)$ be a {\em weight system} such that $\delta=(\delta_1,\cdots,\delta_N)\in\mathbb{Z}_+^N$ is a primitive $N$-tuple with $w_i:={d/\delta_i}\in \mathbb{Z}_+$.
We say the system is of \emph{Calabi-Yau (CY) type} if
\begin{equation}\label{eqnCYcondition}
d=\delta_1+\cdots+\delta_N\,,
\quad
i.e.\,,
\quad
\sum_{i=1}^{N} {1\over w_{i}}=1\,.
\end{equation}
The \emph{dimension} of the CY type weight system $(d; \delta)$ is defined to be  \begin{equation*} \widehat{c}=\sum_{i=1}^N\left(1-{2\delta_i\over d}\right)=N-2\,. \end{equation*}
Let $\mu_d$ be the multiplicative group consisting of $d$-th roots of unity and
\begin{equation*}
J_\delta=(\zeta_d^{\delta_1}, \cdots, \zeta_d^{\delta_N})
\in \mu_d,  \quad \zeta_d:=\exp(2\pi\sqrt{-1}/ d)\,.
\end{equation*}

We call the data $([\mathbb{C}^N/\langle J_\delta\rangle], W)$ a Landau-Ginzburg (LG) space, where
 $W$ is a non-degenerate quasi-homogeneous polynomial on $\mathbb{C}^N$ satisfying 
\begin{equation*}
W(\lambda^{\delta_1}x_1,\cdots,\lambda^{\delta_N}x_N )=\lambda^d W(x_1,\cdots,x_N), \quad \forall \lambda\in \C^*\,.
\end{equation*}
The polynomial $W$ is assumed to
have only an isolated critical point at the origin and not involve quadratic terms $x_ix_j$, $i\neq j$.
In general, we can consider  Landau-Ginzburg spaces $([\mathbb{C}^N/G], W)$ for a group $G$ which is a subgroup of the group of diagonal symmetries with $J_\delta\in G$ (see \cite{Fan:2013, CLL}).
Two enumerative theories can be associated to such a LG space:
\begin{itemize}
\item The Gromov-Witten (GW) theory of the $G/\langle J_\delta\rangle$-quotient of the hypersurface defined by the vanishing of $W$ in the corresponding weighted projective space $\mathbb{P}^{N-1}(\delta_1,\cdots, \delta_N)$. The quotient space is a CY $(N-2)$-orbifold by the CY condition in \eqref{eqnCYcondition}.

\item The Fan-Jarvis-Ruan-Witten (FJRW) theory of the pair $(W,G)$ as introduced in \cite{Fan:book, Fan:2013}.
\end{itemize}

Both the GW theory and the FJRW theory associated to a CY type weight system are Cohomological Field Theories (CohFT, for short) in the sense
of \cite{Kontsevich:1994}.  \\

In this work we shall focus on the theories arising from one-dimensional CY type weight systems.
These systems 
are classified by
\begin{equation}
\label{weight-one}
(d;\delta)=(3;1,1,1),\ (4;1,1,2), \ (6;1,2,3)\,.
\end{equation}
The LG space we consider are $([\mathbb{C}^3/\langle J_\delta\rangle], W)$, with $W$ the
Fermat polynomials
\begin{equation}
\label{fermat-elliptic}
W=x_1^{d/\delta_1}+x_2^{d/\delta_2}+x_3^{d/\delta_3}\,.
\end{equation}

On the CY-side,  
the hypersurface $W=0$ in the weighted projective space $\mathbb{P}^{2}(\delta_1,\delta_2,\delta_3)$ is an elliptic curve, denoted by 
$\E_{d}$ or $\E$ (when the degree $d$ is implicit or unimportant in the discussion) for simplicity. We focus on the GW theory of $\E$.
The GW state space is then defined to be $\mathscr{H}_{\E}:=H^*(\E, \mathbb{C})$.
Let $\overline{\M}_{g,n}(\E, \beta)$ be the moduli stack of degree-$\beta$ stable maps from a connected genus $g$ curve with
$n$ markings to the target $\E$.
Let ${\rm ev}_k,k=1,2,\cdots, n$ be the evaluation morphisms, $\pi$ be the forgetful morphism, and
$[\overline{\M}_{g,n}(\E, \beta)]^{\rm vir}$ be the virtual fundamental cycle of $\overline{\M}_{g,n}(\E, \beta)$.
 The  \emph{ancestor GW invariants} are  given by
\begin{equation*}
\LD\alpha_1\psi_1^{\ell_1},\cdots,\alpha_n\psi_n^{\ell_n}\RD^{\E}_{g,n,\beta}
=\int_{[\overline{\M}_{g,n}(\E,\beta)]^{\rm vir}}\prod_{k=1}^n{\rm ev}_k^*(\alpha_k)\pi^*\psi_k^{\ell_k}\,.
\end{equation*}
The \emph{ancestor GW correlation function} is the formal $q$-series
 \begin{equation}
 \label{GW-function}
 \LL\alpha_1\psi_1^{\ell_1},\cdots,\alpha_n\psi_n^{\ell_n}\RR_{g,n}^{\E}(q)
 =\sum_{d\geq0}q^\beta \LD\alpha_1\psi_1^{\ell_1},\cdots,\alpha_n\psi_n^{\ell_n}\RD_{g,n,\beta}^{\E}\,.
 \end{equation}
By the virtual degree counting of $[\overline{\M}_{g,n}(\E, \beta)]^{\rm vir}$, if the series
$$\LL\alpha_1\psi_1^{\ell_1},\cdots,\alpha_n\psi_n^{\ell_n}\RR_{g,n}^{\E}(q)$$ in \eqref{GW-function} is nontrivial, then
 \begin{equation}
 \label{gw-degree}
 \sum_{k=1}^{n}\left({\deg\alpha_k\over 2}+\ell_k\right)=(3-\dim_\C \E)(g-1)+n=2g-2+n\,.
 \end{equation}


On the LG-side, we consider the FJRW theory of the pair $(W, \langle J_\delta\rangle)$ as originally constructed in  \cite{Fan:book, Fan:2013}.
The main ingredients consist of a CohFT 
$$\left(\mathscr{H}_{(W, \langle J_\delta \rangle)}, \langle , \rangle, \one, \Lambda^{(W,  \langle J_\delta\rangle )}\right)$$ and FJRW invariants (see Section \ref{cohft-def} for details)
$$\LD\alpha_1\psi_1^{\ell_1},\cdots,\alpha_n\psi_n^{\ell_n}\RD_{g,n}^{(W, \langle J_\delta \rangle)},$$  with $\alpha_i$ elements in the vector space $\mathscr{H}_{(W,\langle J_\delta)\rangle}$.
The space $\mathscr{H}_{(W,\langle J_\delta\rangle)}$ contains a canonical degree-$2$ element, denoted by $\phi$ below.
We 
assemble the FJRW invariants into an \emph{ancestor FJRW correlation function} (as a formal series in $s$)
 \begin{equation}
 \label{FJRW-function}
 \LL\alpha_1\psi_1^{\ell_1},\cdots,\alpha_n\psi_n^{\ell_n}\RR_{g,n}^{(W, \langle J_\delta\rangle)}(s)
 :=\sum_{m=0}^{\infty}{1\over m!} \LD \alpha_1\psi_1^{\ell_1},\cdots,\alpha_n\psi_n^{\ell_n}, \underbrace{s\phi, \cdots, s\phi}_{m}\RD_{g,n+m}^{(W,\langle J_\delta\rangle )}.
 \end{equation}

\subsection{LG/CY correspondence via  modularity}

One of the motivation in constructing the FJRW invariants \cite{Fan:book, Fan:2013}
is to understand mathematically the so-called {\em Landau-Ginzburg/Calabi-Yau correspondence} proposed by physicists \cite{Vafa:1989, Greene:1989, Martinec:1990, Witten:1993}.
The \emph{Landau-Ginzburg/Calabi-Yau Correspondence Conjecture} \cite{Fan:2013, Chiodo:2011-b, Ruan:2012} predicts that
for a CY type weight system the corresponding GW theory and the FJRW theory
are related.
In the past decade, a lot of effort has been made to formulate and solve this conjecture:
\begin{itemize}
\item
An LG/CY correspondence between the vector spaces is solved in \cite{Chiodo:2011}.
\item
Genus-zero LG/CY correspondence for various pairs $(W, G)$ have been studied using Givental's I-functions, see \cite{Chiodo:2010,  Priddis:2013, Chiodo:2014, Priddis:2014, Clader:2017, Basalaev:2016}.
\item 
For the quintic 3-fold, the correspondence has been pushed to genus one \cite{Guo:2019}.
\item 
For higher genus, the only known examples in the work \cite{Krawitz:2011, Milanov:2011, Milanov:2016, Shen:2016, Iritani:2016} are all generically semisimple and therefore the correspondence at higher genus is a consequence of the genus-zero correspondence, based on Givental-Teleman's classification of semisimple CohFTs \cite{Giv01a, Teleman:2012}.
\end{itemize}

One of the main results of the present work is to solve this conjecture for the Fermat cubic pair $(W=x_1^3+x_2^3+x_3^3, J_\delta)$ at all genus, using the properties of moduli spaces and quasi-modular forms.
We remark that the GW CohFT and the FJRW CohFT for such a pair are not generically semisimple and therefore this case is beyond the scope of Givental-Teleman's results.

\subsubsection{Quasi-modular forms and Chazy equation}
\label{secintrochazy}
Specializing to the cases of one-dimensional CY type weight systems,
it is known \cite{Bloch:2000, Okounkov:2006} that the GW correlation functions
for an elliptic curve are quasi-modular forms \cite{Kaneko}.
The key of this work is to relate the generating series in \eqref{GW-function} and \eqref{FJRW-function} using transformations on quasi-modular forms.

Consider
the Eisenstein series
\begin{equation}
E_{2k}(\tau):={1\over 2\zeta(2k)}\sum_{\substack{c, d\in\mathbb{Z}\\ (c,d)=1}}{1\over (c\tau+d)^{2k}}\,, \quad \tau\in\mathbb{H}\,,
\end{equation}
where $\zeta(2k)$ are the zeta-values.
These are holomorphic functions on the upper-half plane $\mathbb{H}$, of which
$E_{2k}$, $k\geq 2$, are modular under the group
$\Gamma:={\rm SL}(2,\mathbb{Z})/\{\pm1\}$; while $E_{2}$ is
 \emph{quasi-modular} \cite{Kaneko}.
To be more precise, $E_{2}$
 is not modular, but its non-holomorphic modification $\widehat{E}_2(\tau, \bar\tau)$ is modular where
\begin{equation*}
\widehat{E}_2(\tau, \bar\tau):=E_2(\tau)-{3\over \pi\, {\rm Im}(\tau)}\,.
\end{equation*}
The set of quasi-modular forms (we regard modular forms as special cases of quasi-modular forms) for
$\Gamma$ form a ring  \cite{Kaneko}.
\begin{equation}
\widetilde{M}_*(\Gamma):=\mathbb{C}[E_2(\tau), E_4(\tau), E_6(\tau)]\,.
\end{equation}
The set of almost-holomorphic modular forms as introduced \cite{Kaneko}
also gives rise to a ring that is isomorphic to $\widetilde{M}_*(\Gamma)$
\begin{equation}
\widehat{M}_*(\Gamma):=\mathbb{C}[\widehat{E}_2(\tau,\bar{\tau}), E_4(\tau), E_6(\tau)]\,.
\end{equation}

Let $q=\exp (2\pi\sqrt{-1}\tau)$. The GW invariants of elliptic curves are \cite{Okounkov:2006} Fourier
coefficients expanded around the infinity cusp $\tau=\sqrt{-1}\infty$ of certain quasi-modular forms. 
For example\footnote{We
are sometimes sloppy about the argument for a quasi-modular form when no confusion should arise. For instance we shall occasionally write $E_k(q)$ for  $E_k(\tau)$.},
let $\omega\in H^2(\E)$ be the Poincar\'e dual of the point class, then
\begin{equation}
\label{genus-one-elliptic}
-24\LL\omega\RR_{1,1}^{\E}(q)= E_2(q)=1-{1\over 24}\sum_{n=1}^{\infty}n{q^n\over 1-q^n}\,.
\end{equation}
For any $f\in \widetilde{M}_*(\Gamma)$, we define
\begin{equation*}
f'(\tau):={1\over 2\pi\sqrt{-1}}\cdot {df\over d\tau}\,.
\end{equation*}
The Eisenstein series $E_2$, $E_4$, and $E_6$ satisfy the so-called Ramanujan identities
\begin{equation}
\label{ramanujan}
E_2'={E_2^2-E_4\over 12}\,, \quad E_4'={E_2E_4-E_6\over 3}\,, \quad E_6'={E_2E_6-E_4^2\over 2}\,.
\end{equation}
Eliminating $E_{4},E_{4}$, we see that
$E_2$ is a solution to 
 the so-called Chazy equation,
\begin{equation}
\label{chazy}
2f'''-2f\cdot f''+3(f')^2=0\,.
\end{equation}

Our key observation is that the Chazy equation \eqref{chazy} appears in both GW/FJRW theory for one-dimensional CY weight systems, thanks to the Belorousski-Pandharipande relation discovered in \cite{Belorousski:2000}.
\begin{prop}
\label{main-lemma1}
Consider the LG space $([\mathbb{C}^3/\langle J_\delta\rangle], W)$ given by \eqref{weight-one} and \eqref{fermat-elliptic}.
Then both  the genus-one GW correlation function $-24\LL\omega\RR_{1,1}^{\E}(q)$  and  the genus-one FJRW correlation function
$-24\LL\phi\RR^{(W, \langle J_\delta\rangle)}_{1,1}(s)$ are solutions to the Chazy equation \eqref{chazy}.
\end{prop}
Here for a function $f(q)$ in $q$, we use the convention $f'(q)=q\partial_q f$;  for a function $f(s)$ in $s$, $f'(s)=\partial_s f$.

Furthermore, using more tautological relations discovered in \cite{Ionel: 2002, Faber: 2005}, we can show that both the GW and FJRW correlation functions
in \eqref{GW-function} and \eqref{FJRW-function} are determined by the genus-one correlation functions in Proposition \ref{main-lemma1}.
 \begin{prop}\label{main-lemma2}
Consider the LG space $([\mathbb{C}^3/\langle J_\delta\rangle], W)$ given by \eqref{weight-one} and \eqref{fermat-elliptic}.
Let
\begin{equation*}
f=-24\LL\omega\RR_{1,1}^{\E}\quad \text{or}\quad = -24\LL\phi\RR^{(W, \langle J_\delta\rangle)}_{1,1}\,.
\end{equation*}
then the GW correlation functions  in \eqref{GW-function} (or  the FJRW correlation functions  in \eqref{FJRW-function}) are
determined from $f$ by tautological relations and are
elements in the ring $\mathbb{C}[f, f', f'']$.
\end{prop}

\subsubsection{LG/CY correspondence via Cayley transformation}

By direct calculation, we can show $\LL\omega\RR_{1,1}^{\E}(q)$ and $\LL\phi\RR^{(W, \langle J_\delta\rangle)}_{1,1}(s)$ are expansions of the same quasi-modular form $-(1/24)\cdot E_2(\tau)$
at two different points on the upper-half plane. In particular, the GW functions are Fourier expansion around the cusp $\tau=\sqrt{-1}\infty$.
This viewpoint allows us to relate  the GW functions  in \eqref{GW-function} and  the FJRW functions  in \eqref{FJRW-function} by a variant of the Cayley transformation
which we now briefly review following
\cite{Shen:2016}.

For any point $\tau_*\in \mathbb{H}$, there exists a Cayley transform that maps a point $\tau$ on the upper half-plane $\mathbb{H}$ to a point $s(\tau)$ in the unit disk $\mathbb{D}$, that is,
\begin{equation*}
s(\tau)=(\tau_*-\bar{\tau}_*){\tau-\tau_*\over \tau-\bar{\tau}_*}\,.
\end{equation*}
This transform is biholomorphic and we denote its inverse by $\tau(s)$.
Following \cite{Zagier:2008} and \cite{Shen:2016}, there exists a Cayley transformation 
that maps a weight-$k$ almost-holomorphic modular form
$$\widehat{f}\in \widehat{M}_*(\Gamma)=\mathbb{C}[\widehat{E}_2(\tau,\bar\tau), E_4(\tau), E_6(\tau)]$$ 
to
\begin{equation}
\label{automorphic-factor}
\left({\tau(s)-\bar\tau_{*} \over \tau_{*}-\bar\tau_{*} }\right)^k \cdot  \widehat{f}\left(\tau(s),\overline{\tau(s)}\right)\,.
\end{equation}
The Taylor expansion of the image gives a natural way to expand the almost-holomorphic modular form $\hat{f}$ near $\tau=\tau_*$, where the local complex coordinate is $s(\tau)$.

Using the fact that the two rings $\widetilde{M}_*(\Gamma)$ and $\widehat{M}_*(\Gamma)$ are isomorphic differential ring, a {\em holomorphic Cayley transformation} $\mathscr{C}_{ \tau_* }^{\rm hol}$ (see Section \ref{secLGCY})
can then be defined  \cite{Shen:2016}.
This turns out to be the correct transformation that relates the GW correlation functions  in \eqref{GW-function} and  the FJRW correlation functions  in \eqref{FJRW-function}, both of which are holomorphic. 
It allows us to solve the LG/CY Correspondence Conjecture for the Fermat cubic pair. 
\begin{thm}
\label{main-theorem}
Consider the Fermat cubic polynomial $W=x_1^3+x_2^3+x_3^3$ and the LG space $([\mathbb{C}^3/ \mu_3], W)$. There exists a degree- and grading-preserving vector space isomorphism
\begin{equation*}
\Psi: \mathscr{H}_{\E}=H^*(\E)\longrightarrow \mathscr{H}_{(W,\mu_3)}
\end{equation*}
and a holomorphic Cayley transformation $\mathscr{C}_{ \tau_* }^{\rm hol}$ with 
$$\tau_*
=-{\sqrt{-1}\over \sqrt{3}} \exp({2\pi \sqrt{-1}\over 3})
\in\mathbb{H}\,,$$
such that
\begin{equation*}
\mathscr{C}_{ \tau_* }^{\rm hol}\left(\LL\alpha_1\psi_1^{\ell_1}, \cdots, \alpha_n\psi_n^{\ell_n}\RR_{g,n}^{\E}(q)\right)
=\LL\Psi(\alpha_1)\psi_1^{\ell_1},\cdots, \Psi(\alpha_n)\psi_n^{\ell_n}\RR_{g,n}^{(W, \mu_3) }(s)\,.
\end{equation*}
\end{thm}
The explicit construction of $\Psi$ and $\mathscr{C}_{ \tau_* }^{\rm hol}$ will be given in Section \ref{secLGCY}.

Theorem \ref{main-theorem} can be generalized to 
the rest of the one-dimensional CY type weight systems in \eqref{weight-one} straightforwardly:
the only difference lies in the technical computations on the initial genus-one FJRW invariants. 
This approach of using modular forms was previously introduced in \cite{Shen:2016} for elliptic orbifold curves.

It is worthwhile to mention that for one-dimensional CY type weight systems, our approach of the LG/CY correspondence is compatible with the I-function approach introduced in \cite{Chiodo:2010, Milanov:2011}.
In fact, 
the automorphy factor in the Cayley transformation \eqref{automorphic-factor} provides the equivalent information as the symplectic
transformation that appears in \cite[Corollary 4.2.4]{Chiodo:2010}.

\subsection{Applications: higher-genus FJRW invariants and their structures}

The higher-genus FJRW invariants are very difficult to compute in general.
In our example, 
with the identification of the correlation functions with quasi-modular forms, various results from the GW-side can be transformed
into the LG-side, by the virtue of the holomorphic Cayley transformation
which respects the differential ring structure of quasi-modular forms.
In  particular,  higher-genus 
FJRW invariants can be computed easily and nice structures of the FJRW correlation functions can be obtained for free.

Indeed,  higher-genus FJRW invariants are determined from the results on descendent GW invariants of elliptic curves
given by Bloch-Okounkov \cite{Bloch:2000}, whose generating series admit very concrete and beautiful formulae.
The following gives a sample of the computations.
\begin{cor}\label{coronepointFJRW}
For the $d=3$ case, the following holds for the ancestor FJRW correlation functions
\begin{equation*}
\LL  \phi\psi_1^{2g-2}\RR_{g,1}^{(W, \mu_3)}= \sum_{\substack{\ell, m,n\geq 0\\
\ell+2m+3n=g}}
{ b_{m,n}\over \ell!} \left(-{\mathscr{C}_{ \tau_{*}}^{\mathrm{hol}}(E_2)\over 24} \right)^{\ell}   \left({ \mathscr{C}_{ \tau_{*}}^{\mathrm{hol}}(E_4)\over 24}\right)^{m} \left( - { \mathscr{C}_{ \tau_{*}}^{\mathrm{hol}}(E_6)\over 108}\right)^{n}
 \,,
\end{equation*}
where $\mathscr{C}_{ \tau_{*}}^{\mathrm{hol}}(E_{2i}),i=1,2,3$ are holomorphic Cayley transformations of the Eisenstein series $E_2,E_4,E_6$
whose expansions can be computed explicitly, while $\{b_{m,n}\}_{m,n}$ are rational numbers that can be obtained recursively.
\end{cor}

The holomorphic anomaly equations (HAE) discovered in \cite{OP18}
and the Virasoro constraints discovered in \cite{Okounkov-vir}
 for the GW theory of elliptic curves also carry over to the corresponding FJRW theory.
See Corollary \ref{hae-lg} and Corollary \ref{virasoro-lg} for the explicit statements.
 \\

\paragraph*{\bf Plan of the paper}

In Section \ref{secreconstruction} we review the basic construction of CohFTs,
and use tautological relations in particular the Belorousski-Pandharipande relation to prove
Proposition \ref{main-lemma1} and
Proposition \ref{main-lemma2}.
In Section \ref{secgenusoneinitialvalues} we calculate a genus-one FJRW invariant for the $d=3$ case
using cosection localization.
In Section \ref{secLGCY} we prove Theorem \ref{main-theorem} using properties of quasi-modular forms.
In Section \ref{secapplications1}
we review some results on GW invariants for the elliptic curve and discuss the 
ancestor/descendent correspondence.
 In Section  \ref{secapplications2} we give some applications of the quasi-modularity of the GW and FJRW theory for the $d=3$ case, 
 such as  the explicit computations of higher-genus FJRW invariants basing on the results on the GW invariants of the elliptic curve, 
 the derivation of holomorphic anomaly equations and Virasoro constraints they satisfy.\\

\paragraph*{\bf Acknowledgement}
Y. Shen would like to thank Qizheng Yin, Aaron Pixton, and Felix Janda for inspiring discussions on tautological relations.
J. Zhou thanks Baosen Wu and Zijun Zhou for useful discussions.

J. Li is partially supported by National Natural Science Foundation of China no. 12071079.
Y. Shen is partially supported by Simons Collaboration Grant 587119.
J. Zhou is supported by a start-up grant at Tsinghua University, the Young overseas high-level talents introduction plan of China, and
the national key research and development program of China (No. 2020YFA0713000).
Part of J. Zhou's work was done while he was a postdoc at the Mathematical Institute of University of Cologne and was partially supported by German Research Foundation Grant CRC/TRR 191.

\section{Belorousski-Pandharipande relation and Chazy equation}
\label{secreconstruction}

We study the two Cohomological Field Theories (both GW theory and FJRW theory) for the one-dimensional CY type weight systems using tautological relations and axioms of CohFTs. The key is the identification between
Belorousski-Pandharipande relation and Chazy equation.

\subsection{Cohomological field theories}\label{cohft-def}
Both the GW theory and FJRW theory of the LG space $([\mathbb{C}^N/G], W)$ satisfy axioms of Cohomological Field Theories (CohFT) in the sense
of \cite{Kontsevich:1994}, which we briefly recall now.

Let ${\overline{\mathcal{M}}_{g,n}}$ be the Deligne-Mumford moduli stack of genus
$g$ stable (i.e., $2g-2+n>0$) curves with $n$ markings.
A \emph{Cohomological Field Theory with a flat identity} is a quadruple
\begin{equation*}
(\mathscr{H}, \eta, {\bf 1}, \Lambda)\,,
\end{equation*}
where the \emph{state space}
\begin{equation*}
\mathscr{H}:=\mathscr{H}^{\rm even}\bigoplus\mathscr{H}^{\rm odd}
\end{equation*}
is a
$\mathbb{Z}_2$-graded finite dimensional $\C$-vector space (called superspace in \cite{Kontsevich:1994}),
$\eta$ is a non-degenerate pairing on $\mathscr{H}$,
$\one\in\mathscr{H}$ is the \emph{flat identity},
and
\begin{equation*}
\Lambda:=\left\{\Lambda_{g,n}\in {\rm Hom}\big(\mathscr{H}^{\otimes n}, H^*({\overline{\mathcal{M}}_{g,n}},\C)\big) \right\}
\end{equation*}
is a set of multi-linear maps satisfying the CohFT axioms below:
\begin{enumerate}
\item [(i)] Let $|\cdot|$ be the grading. The maps $\Lambda_{g,n}$ satisfy
\begin{equation}
\label{Z2-grading}\Lambda_{g,n}(\cdots, \alpha_1,\alpha_2, \cdots)=(-1)^{|\alpha_1|\cdot|\alpha_2|}\Lambda_{g,n}(\cdots, \alpha_2,\alpha_1, \cdots)\,.
\end{equation}
\item [(ii)] The maps in $\Lambda$ are compatible with the gluing and the forgetful morphisms
\begin{itemize}
\item $\overline{\mathcal{M}}_{g_1,n_1+1}\times \overline{\mathcal{M}}_{g_2,n_2+1}\to \overline{\mathcal{M}}_{g,n}$ and
$\overline{\mathcal{M}}_{g-1,n+2}\to \overline{\mathcal{M}}_{g,n}$; 
\item $\pi: \overline{\mathcal{M}}_{g,n+1}\to \overline{\mathcal{M}}_{g,n}$ forgetting one of the markings.
\end{itemize}
For example, the compatibility with the forgetting morphism is
\begin{equation}
\label{string}
\Lambda_{g,n+1}(\alpha_1, \cdots, \alpha_n, \one)=\pi^*\Lambda_{g,n}(\alpha_1, \cdots, \alpha_n)\,.
\end{equation}
\item [(iii)]
The pairing $\eta$ is compatible with $\Lambda_{0,3}$:
\begin{equation*}
\int_{\overline{\mathcal{M}}_{0,3}}\Lambda_{0,3}(\alpha_1,\alpha_2, \one)=\eta(\alpha_1,\alpha_2)\,.
\end{equation*}
\end{enumerate}

Let $\psi_k\in H^2(\overline{\mathcal{M}}_{g,n})$ be the cotangent line class at the $k$-th marking. For each CohFT $(\mathscr{H}, \eta, {\bf 1}, \Lambda)$,
one defines the quantum invariants from $\Lambda$ by
\begin{equation}
\label{quantum-invariant}
\LD\alpha_1\psi_1^{\ell_1},\cdots,\alpha_n\psi_n^{\ell_n}\RD_{g,n}^{\Lambda}
:=\int_{\overline{\mathcal{M}}_{g,n}}\Lambda_{g,n}(\alpha_1,\cdots,\alpha_n)\prod_{k=1}^n\psi_k^{\ell_k}, \quad \alpha_k\in\mathscr{H}\,.
\end{equation}
Such invariants are called the ancestor GW invariants for the GW CohFT and FJRW invariants for the LG CohFT.
Our focus is the relation between these two types of invariants arising from the same CY type LG space $([\mathbb{C}^N/G], W)$.

Fix a basis $\mathcal{B}$ for  $\mathscr{H}$. 
It is convenient to choose the elements $\alpha_k$ from $ \mathcal{B}$
and parametrize $\alpha_k$ by $s_k$. We introduce the {\em genus-zero primary potential of the CohFT} as a formal power series
\begin{equation}
\label{primary-potential}
\mathcal{F}_0^{\Lambda}:=\sum_{n\geq 0}\sum_{\alpha_k\in \mathcal{B}}{1\over n!}
\LD\alpha_1,\cdots,\alpha_n\RD_{0,n}^{\Lambda}\prod_{k=1}^{n}s_k\,.
\end{equation}
Here primary means all $\ell_k=0$ in \eqref{quantum-invariant}.

\subsubsection{FJRW invariants}

The CohFTs arising from GW theories have become a familiar topic since \cite{Kontsevich:1994}.
Here we only recall some basics on the LG CohFT constructed from the FJRW invariants defined in \cite{Fan:book, Fan:2013}. 
See also \cite{CLL, PV, KL, CKL} for various CohFT constructions for LG models.

As $G$ acts on $\C^N$, for any $\gamma\in G$, the fixed-point set ${\rm Fix}(\gamma)$ is an
$N_\gamma$-dimensional subspace of $\C^N$.
Let $W_{\gamma}$ be the restriction of $W$ on ${\rm Fix}(\gamma)$. Following \cite{Fan:2013},
one considers the graded vector space (called the FJRW state space)
\begin{equation}
\mathscr{H}_{(W,G)}=\bigoplus_{\gamma\in G} \mathscr{H}_{\gamma}\,,
\end{equation}
where each $\mathscr{H}_{\gamma}$ is the space of G-invariants of the middle-dimensional relative cohomology in ${\rm Fix}(\gamma)$.
There is a natural pairing $\langle , \rangle$ and an isomorphism
(see \cite[Section 5.1]{Fan:2013})
\begin{equation}
\label{residue-pairing}
\Big(\mathscr{H}_{(W,G)}, \langle , \rangle\Big)\cong
\Big( \bigoplus_{\gamma\in G}\left({\rm Jac}(W_\gamma)\Omega_{{\rm Fix}(\gamma)}\right)^{G}\,,~ {\rm Res}\Big)\,.
\end{equation}
Here ${\rm Jac}(W_\gamma)$ is the \emph{Jacobi algebra} of $W_\gamma$, $\Omega_{{\rm Fix}(\gamma)}$ is the standard
holomorphic volume form on ${\rm Fix}(\gamma)$ and ${\rm Res}$ is the residue pairing.

In \cite{Fan:book, Fan:2013}, Fan-Jarvis-Ruan constructed the virtual fundamental cycle over the moduli space of $W$-spin structures, and
a corresponding CohFT $$\left(\mathscr{H}_{(W,G)}, \langle , \rangle, \one, \Lambda^{(W, G)}\right).$$
This CohFT defines the so-called \emph{FJRW invariants} $\LD\alpha_1\psi_1^{\ell_1},\cdots,\alpha_n\psi_n^{\ell_n}\RD_{g,n}^{(W, G)}$ through \eqref{quantum-invariant}.
\\

We now specialize to a pair $(W, G)$ given in \eqref{fermat-elliptic} with  $G=\langle J_\delta\rangle$.  
For a set of homogeneous elements $\alpha_k\in\mathscr{H}_{\gamma_k},k=1,2,\cdots n$,
the dimension formula in \cite[Theorem 4.1.8]{Fan:2013} shows
if $\LD\alpha_1\psi_1^{\ell_1},\cdots,\alpha_n\psi_n^{\ell_n}\RD_{g,n}^{(W,\langle J_\delta\rangle)}$ is non-trivial, then
\begin{equation}
\label{lg-degree}
2g-2+n=\sum_{k=1}^{n}{\deg\alpha_k\over 2}+\sum_{k=1}^{n}\ell_k\,.
\end{equation}
We remark that both $\mathscr{H}_{J_{\delta}}$ and $\mathscr{H}_{J_{\delta}^{-1}}$ are one-dimensional: $\mathscr{H}_{J_{\delta}}$ is spanned
by the flat identity $\one\in\mathscr{H}_{J_{\delta}}$
and $\mathscr{H}_{J^{-1}_{\delta}}$ by a canonical degree-$2$ element
$\phi\in \mathscr{H}_{J^{-1}_{\delta}}$. We let $s$ be the corresponding linear coordinate of
the space $\mathscr{H}_{J^{-1}_{\delta}}$. The constraint \eqref{lg-degree} allows us to define the following
 \emph{ancestor FJRW correlation function} (as a formal series in $s$)
 \begin{equation}
 \LL\alpha_1\psi_1^{\ell_1},\cdots,\alpha_n\psi_n^{\ell_n}\RR_{g,n}^{(W,\langle J_\delta\rangle)}(s)
 :=\sum_{m=0}^{\infty}{1\over m!} \LD \alpha_1\psi_1^{\ell_1},\cdots,\alpha_n\psi_n^{\ell_n}, \underbrace{s\phi, \cdots, s\phi}_{m}\RD_{g,n+m}^{(W, \langle J_\delta\rangle)}.
 \end{equation}

In the following, we will use the subscript $"d"$ to label the CY type weight systems in \eqref{weight-one}.
Let $\Omega=dx_1\wedge dx_2\wedge dx_3$. For each polynomial $W_d$, when $d=3$ (resp. $4$; resp.  $6$), we consider the following element
\begin{equation}
h(W_d)=x_1x_2x_3/27 \ (\text{resp.}\ x_1^2x_2^2/32;\,   \text{resp.}\ x_1^4x_2/36)\,.
\end{equation}
 According to \eqref{residue-pairing}, the FJRW state space is
 \begin{equation}
 \label{cubic-basis}
 \mathscr{H}_{(W_d,G_d)}=\mathscr{H}_{J_{\delta}}\bigoplus\mathscr{H}_{J_{\delta}^{-1}}\bigoplus\mathscr{H}_{1\in G_d}=\C\{\one, \phi, \mathfrak{b}_1, \mathfrak{b}_2\}\,.
  \end{equation}
Here the even part is spanned by $\one\in \mathscr{H}_{J_{\delta}}$ and $\phi\in \mathscr{H}_{J_{\delta}^{-1}}$;
while the odd part is spanned by
 \begin{equation*}
 \mathfrak{b}_1=h(W_d)\Omega, \quad  \mathfrak{b}_2=\Omega \in \left({\rm Jac}(W_{d})\Omega\right)^{G}\subseteq \mathscr{H}_{1\in G_d}\,.
   \end{equation*}
The degrees are
\begin{equation}
\label{element-degree}
\deg\one=0\,, \quad \deg\mathfrak{b}_1=\deg\mathfrak{b}_2=1\,, \quad \deg\phi=2\,.
\end{equation}

\subsubsection{Genus-zero comparison}
We begin with a comparison between the genus-zero parts of the two theories.
On the GW-side, recall the state space for the elliptic curve $\E_{d}$ is $H^*(\E_{d}, \mathbb{C})$.
Let ${\bf 1}\in H^0$ be the identity of the cup product, and
$\omega\in H^2$ be the Poincar\'e dual of the point class.
We choose a symplectic basis $\{e_1, e_2\}$
of $H^1$ such that 
\begin{equation*}
e_{1}\cup e_{2}=-e_{2}\cup e_{1}=\omega\,.
\end{equation*}
We define a linear map
$\Psi: H^*(\E_d)\to \mathscr{H}_{(W_d, \langle J_\delta\rangle)}$ by
\begin{equation}
\label{ring-iso}
\Psi(\one)=\one\,, \quad \Psi(\omega)=\phi\,, \quad \Psi(e_{i})=\mathfrak{b}_i\,, \quad i=1,2\,.
\end{equation}

Let $(t_{0},t_{1},t_{2},t)$ be the coordinates with respect to the basis $\{\one, e_{1}, e_{2},\omega\}$.
Similarly we let $(u_{0},u_{1},u_{2},u)$ be the coordinates
with respect to the basis $\{\one, \mathfrak{b}_1, \mathfrak{b}_2 ,\phi\}$.

The moduli stack $\overline{\M}_{g,n}(\E_d,\beta)$ is empty when $g=0$ and $\beta>0$.
Then according to \eqref{primary-potential}, the \emph{genus-zero primary GW potential} is
\begin{equation*}
\mathcal{F}^{\E_{d}}_0={1\over 2}\cdot t_{0}^2t+t_0 t_{1}t_{2}\,.
\end{equation*}
A calculation on residue shows that
\begin{equation}
\label{pairing-with-sign}
\LD\one, \one, \phi\RD_{0,3}^{W_d}=\LD\one, \mathfrak{b}_1, \mathfrak{b}_2\RD_{0,3}^{W_d}=1, \quad \LD\one, \mathfrak{b}_2, \mathfrak{b}_1\RD_{0,3}^{W_d}=-1.
\end{equation}
Thus the \emph{genus-zero primary FJRW potential} is 
\begin{equation*}
\mathcal{F}^{W_d}_0={1\over 2}\cdot u_{0}^2u+u_0u_{1}u_{2}+\text{quantum corrections}\,.
\end{equation*}
These quantum corrections 
vanish as shown below.
This was firstly observed by Francis \cite[Section 4.2]{Francis:2015} using WDVV equations.

\begin{prop}
\label{genus-zero-vanishing}
The map $\Psi$ in \eqref{ring-iso} is a degree- and grading-preserving ring isomorphism, and
\begin{equation}
\label{primary-fjrw-potential}
\mathcal{F}^{W_d}_0={1\over 2}\cdot  u_{0}^2u+u_0u_{1}u_{2}\,.
\end{equation}
\end{prop}
\begin{proof}
It is easy to see $\Psi$ preserves the degree and grading.
To show $\Psi$ is a ring isomorphism, it is enough to prove \eqref{primary-fjrw-potential}.
The compatibility condition \eqref{string} implies the String Equation in FJRW theory. Combining the degree constraints \eqref{element-degree} and \eqref{lg-degree}, we find that the
 quantum corrections are encoded in $C_i(s)$,
where $C_i(s)$ is the correlation function with $i$ copies of $\mathfrak{b}_1$-insertions and $(4-i)$ copies of $\mathfrak{b}_2$-insertions. For example,
\begin{equation*}
C_0(s)=\LL  \mathfrak{b}_1,\mathfrak{b}_1, \mathfrak{b}_1,\mathfrak{b}_1 \RR_{0,4}^{W_d}, \quad C_3(s)=\LL  \mathfrak{b}_1,\mathfrak{b}_2, \mathfrak{b}_2,\mathfrak{b}_2 \RR_{0,4}^{W_d}\,.
\end{equation*}
The $\mathbb{Z}_2$-grading \eqref{Z2-grading} shows $C_i(s)=0$ because for
$\alpha=\mathfrak{b}_1$ or $\mathfrak{b}_2$
\begin{equation*}
\LL\alpha,\alpha,\cdots\RR_{g,n}^{W_d}=(-1)^{|\alpha|\cdot |\alpha|}\LL\alpha,\alpha,\cdots\RR_{g,n}^{W_d}=-\LL\alpha,\alpha,\cdots\RR_{g,n}^{W_d}\,.
\end{equation*}
This proves the claim.
\end{proof}

\subsection{Belorousski-Pandharipande's relation and $g$-reduction}

The tautological rings $RH(\overline{\mathcal{M}}_{g,n})$ of $\overline{\mathcal{M}}_{g,n}$ are defined (see \cite{Faber: 2005} for example) as the smallest system of subrings of $H^*({\overline{\mathcal{M}}_{g,n}})$ stable under push-forward and
pull-back by the gluing and forgetful morphisms.
Thus pulling back the tautological relations in $RH(\overline{\mathcal{M}}_{g,n})$ via the CohFT maps $\Lambda_{g,n}$ gives relations among quantum invariants.
We use this technique to prove Proposition \ref{main-lemma1} and Proposition \ref{main-lemma2}.

\subsubsection{Belorousski-Pandharipande's relation for a genus-one correlation function}
The degree constraints \eqref{element-degree} and \eqref{lg-degree} show that the non-vanishing genus-one primary FJRW invariants could only come from the coefficients in $\LL\phi\RR^{W_d}_{1,1}(s)$.
We determine this series and the GW correlation function $\LL\omega\RR^{\E_d}_{1,1}(q)$ up to some initial values,
using the tautological relation found by Belorousski and Pandharipande  \cite[Theorem 1]{Belorousski:2000}.
The relation is a nontrivial rational equivalence among codimension-$2$ descendent stratum classes in $\overline{\M}_{2,3}$ shown in Figure \ref{B&P-relation} below. 

\begin{figure}[h]
\centering
\includegraphics[width=1\textwidth]{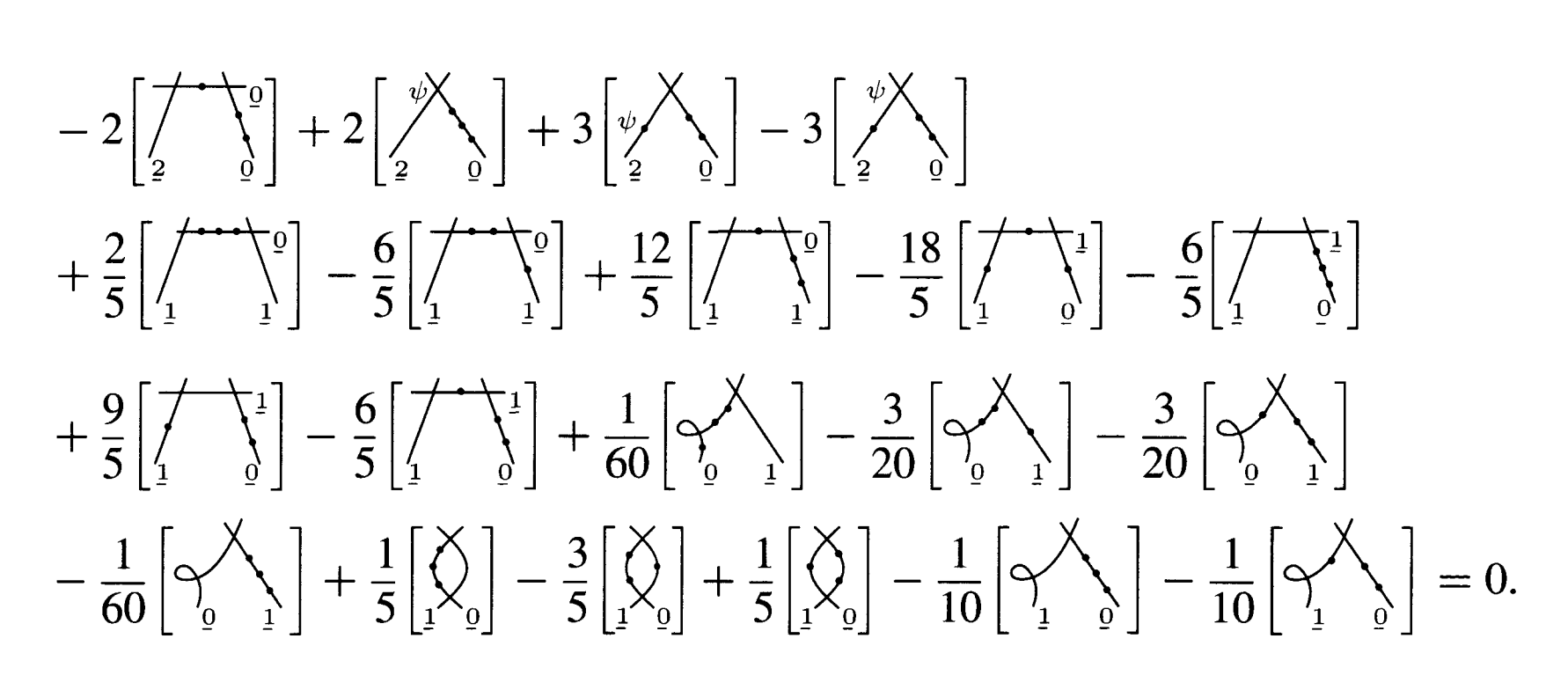}
\caption[Belorousski-Pandharipande relation]{Belorousski-Pandharipande relation. Taken from \cite[Theorem 1, formula (4)]{Belorousski:2000}.}
\label{B&P-relation}
\end{figure}

Each stratum in the relation is represented by the topological type of the stable curve corresponding to the generic moduli point in the stratum. The markings on the stratum are unassigned. The geometric genera of the components are underlined.  The cotangent line class $\psi$ always appears on the genus-$2$ component.

\begin{proof}[Proof of  Proposition \ref{main-lemma1}]
On the FJRW-side, we integrate 
$$\Lambda^{W_d}_{2,3}(\phi,\phi,\phi)\in H^4(\overline{\mathcal{M}}_{2,3})$$ 
over the Belorousski-Pandharipande relation. We read off one term from each stratum.\\

\noindent{\bf Strata in the 1st row of Figure \ref{B&P-relation}.}
Let us consider the 1st stratum in the 1st row. The integration over this stratum gives the term
\begin{equation*}
 -2\sum_{\alpha, \alpha', \beta, \beta'\in\mathscr{H}_{W_d}}\LL \alpha \RR_{2,1}^{W_d}(s) \eta^{\alpha, \alpha'} \LL \alpha', \phi , \beta \RR_{0,3}^{W_d}(s)\eta^{\beta, \beta'}\LL\beta',\phi, \phi \RR_{0,3}^{W_d}(s)\,.
 \end{equation*}
Here the notations
$\eta^{\alpha, \alpha'}$ stands for the $(\alpha,\alpha')$ component of the inverse of the paring $\eta$, etc.
For any homogeneous element $\alpha\in\mathscr{H}_{W_d}$, 
the degree constraint \eqref{lg-degree} implies that if $\LL \alpha \RR_{2,1}^{W_d}(s)$ is nonzero, then we must have
$$2(2-1)+1={\deg\alpha\over 2}\,.$$
This contradicts  \eqref{element-degree}, where we have $\deg\alpha=0, 1, 2$.
Thus we have that  $\LL \alpha \RR_{2,1}^{W_d}(s)=0$ and hence the contribution from this stratum is $0$.
Similar arguments imply that the contribution from all the strata in the 1st row of Figure \ref{B&P-relation} vanish, since the contribution from each stratum must contain one of the following terms as a factor
$$\LL \alpha \RR_{2,1}^{W_d}(s)=\LL \alpha\psi_1 \RR_{2,1}^{W_d}(s)=\LL\phi\psi_1, \alpha \RR_{2,2}^{W_d}(s)=\LL \phi, \alpha\psi_2 \RR_{2,2}^{W_d}(s)=0\,.$$\\

\noindent{\bf Other vanishing strata.}
Now we look at the 1st, 2nd, and 5th stratum in the 2nd row, 
the 3rd, 4th and 5th stratum in the 3rd row, and
the 2nd, 3rd, 5th, 6th stratum in the last row. Each stratum has a genus-zero component with at least $4$ markings (including the nodes). According to Proposition \ref{genus-zero-vanishing}, one has for the primary invariants
$$\LL\cdots\RR_{0,n}^{W_d}=0\,, \quad \forall \,n\geq 4\,.$$
Thus the integration of $\Lambda^{W_d}_{2,3}(\phi,\phi,\phi)\in H^4(\overline{\mathcal{M}}_{2,3})$ over each of these strata vanishes.

For the 1st and 2nd stratum in the 3rd row, the genus-zero component only contains $3$ markings, but at least $2$ of the markings are labeled with the class $\phi$. 
Again by Proposition \ref{genus-zero-vanishing}, we have 
$$\LL\phi, \phi, \alpha\RR_{0,3}^{W_d}=0\,, \quad \forall \,\alpha\in\mathscr{H}_{W_d}\,.$$
So the contribution from these two strata also vanish.

Finally, the integration on the 1st stratum in the 4th row also vanishes. 
This is a consequence of the $\mathbb{Z}_2$-grading.
In fact, we apply the degree constraint \eqref{lg-degree} to the genus-one component and find that the non-vanishing contribution from this stratum, if exists, should be of the form
$$-{1\over 60}\sum_{\alpha, \alpha'}\LL\phi, \phi, \phi, \phi\RR_{1,4}^{W_d}(s)\eta^{\phi, \one}\LL\one, \alpha, \alpha'\RR_{0,3}^{W_d}\eta^{\alpha', \alpha}.$$
The vanishing of this term is a direct consequence of the formula \eqref{pairing-with-sign}, where
$$ \eta^{\phi, \one}=\eta^{\mathfrak{b}_1, \mathfrak{b}_2}=1, \quad\eta^{\mathfrak{b}_2, \mathfrak{b}_1}=-1.$$\\

\noindent{\bf Non-vanishing terms.}
Now we see that all the possibly non-vanishing terms are from the  3rd and 4th stratum in the 2nd row, and the 4th stratum in the last row. 
Let us calculate them term by term. 
The 3rd stratum of the 2nd row gives a possibly non-vanishing term
\begin{equation*}
{12\over 5} \LL \phi \RR_{1,1}^{W_d}(s) \eta^{\phi, \one}\LL \one, \phi , \one \RR_{0,3}^{W_d}(s)\eta^{\one, \phi} \LL \phi, \phi, \phi \RR_{1,3}^{W_d}(s)
={12\over 5}g\cdot g''.
\end{equation*}
The 4th stratum of the 2nd row gives a possibly non-vanishing term
\begin{equation*}
-{18\over 5} \LL \phi, \phi\RR_{1,2}^{W_d}(s) \eta^{\phi, \one}\LL \one, \phi , \one \RR_{0,3}^{W_d}(s)\eta^{\one, \phi} \LL \phi, \phi \RR_{1,2}^{W_d}(s)
=-{18\over 5}g'\cdot g'.
\end{equation*}
The 4th stratum of the last row gives a possibly non-vanishing term
\begin{equation*}
{1\over 5}\cdot {1\over 2}\cdot \LL \one, \phi, \one\RR_{0,3}^{W_d}(s) \eta^{\one, \phi}\LL \phi, \phi, \phi , \phi \RR_{1,4}^{W_d}(s)\eta^{\phi, \one}
={1\over 5}\cdot{g'''\over 2}.
\end{equation*}
Here the  denominator $2$ in the term above comes from the automorphism of the graph.\\

Putting all these together, we see the Belorousski-Pandharipande relation in Figure \ref{B&P-relation} allows us to verify by brute-force computation that the correlation function
$g:=\LL \phi \RR_{1,1}^{W_d}(s)$
is a solution to
\begin{equation}
\label{BP-relation}
{12\over 5}g\cdot g''-{18\over 5}g'\cdot g'+{1\over 5}\cdot {g'''\over{2}}=0\,.
\end{equation}
Thus $-24 \LL \phi \RR_{1,1}^{W_d}(s)$ is a solution of the Chazy equation \eqref{chazy}.

Similarly, by integrating the GW cycle $\Lambda_{2,3}^{\E_d}(\omega,\omega,\omega)$ over
the Belorousski-Pandharipande relation in Figure \ref{B&P-relation}, we see that
 $-24 \LL \phi \RR_{1,1}^{\E_d}(q)$ is a solution of the Chazy
 equation \eqref{chazy}. This completes the proof of Proposition \ref{main-lemma1}.
\end{proof}

The identity \eqref{BP-relation} is independent of the specific form $\E_d$,
as should be the case since the GW invariants are independent of the choice of complex structures put on the elliptic curve.
\begin{rem}
For the elliptic orbifold curve $\cX_{N}:=\mathcal{E}^{(N)}/\mu_{N}$ for some particular elliptic curve $\mathcal{E}^{(N)}$ that admits $\mu_{N}$
as its automorphism group, the first stratum in the fourth line does not vanish. Let $\mu$ be the rank of the Chen-Ruan cohomology $H^*_{\rm CR}(\mathcal{X}_N)$
which satisfies
\begin{equation*}
1-{\mu\over 12}={1\over N}\,.
\end{equation*}
Define similarly $g=\LL \mathcal{P}\RR_{1,1}^{\mathcal{X}_{N}}$ where $\mathcal{P}$ is the point class on $\cX_{N}$. 
The Belorousski-Pandharipande relation now gives
\begin{equation*}
{12\over 5}g\cdot g''-{18\over 5}(g')^2+\left(-{\mu\over60}+{1\over 5}\right){g'''\over 2}=0\,,
\end{equation*}
where $'=Q \partial_{Q}$ is now the derivative with respect to the parameter for the point class $\mathcal{P}$. 
Then $f=-24g$ satisfies
\begin{equation*}
2f\cdot f''-3(f')^2-2\left(1-{\mu\over 12}\right)f'''=0\,.
\end{equation*}
Its solutions coincide with the ones to \eqref{BP-relation} via the relation $Q=q^N$,
see \cite{Shen:2017} for more details.

\end{rem}

\subsubsection{g-reduction for higher-genus correlation functions}

Now we prove Proposition \ref{main-lemma2} using the \emph{$g$-reduction technique} introduced in \cite{Faber: 2010}. We recall the following result. 
\begin{lem}\label{g-reduction}
\cite{Ionel: 2002, Faber: 2005}
Let $M(\psi, \kappa)$ be a monomial of $\psi$-classes and $\kappa$-classes $\overline{\mathcal{M}}_{g,n}$.
Assume $\deg M\geq g$ when $g\geq1$, and $\deg M\geq1$ when $g=0$, then $M(\psi, \kappa)$ is equal to a linear combination
of dual graphs on the boundary of $\overline{\mathcal{M}}_{g,n}$.
\end{lem}

\begin{proof}[Proof of Proposition \ref{main-lemma2}]

Consider the GW or FJRW correlation function of the form
\begin{equation*}
\LL\alpha_1\psi_1^{\ell_1},\cdots,\alpha_n\psi_n^{\ell_n}\RR_{g,n}^\clubsuit \, ,  \quad  \clubsuit=\E_d\ \text{or}\ W_d\,.
\end{equation*}
Using that the cohomology classes have $0\leq \deg\alpha_k\leq2$, and using \eqref{gw-degree} and \eqref{lg-degree},
we deduce that the correlation function is trivial if
\begin{equation*}
\sum_{k=1}^n  \ell_k<2g-2\,.
\end{equation*}
Now we assume it is nontrivial and $\sum_{k=1}^n \ell_k\ge 1$, then we must have
\begin{equation}\label{g-nonvanishing}
\deg\left(\prod_{k=1}^{n}\psi_k^{\ell_k}\right)=
\sum_{k=1}^n \ell_k\geq
\left\{
\begin{array}{ll}
2g-2\geq g,& g\geq2\,,\\
1,& g=0,\, 1\,.\\
\end{array}
\right.
\end{equation}
Then $\prod_{k=1}^{n}\psi_k^{\ell_k}$ is a monomial satisfying the condition in Lemma \ref{g-reduction}, thus we can apply this technique and use
the Splitting Axiom in GW/FJRW theory to rewrite the function as a linear combination of products of other correlation functions, with smaller genera.

We then repeat the process for nontrivial correlation functions with smaller genera and eventually rewrite  the correlation function as a linear combination of products of primary (all $\ell_k=0$) correlation functions in genus-zero
(which are just constants) and in genus-one, which must be $f_d^{(n-1)}=\LL\omega, \cdots, \omega\RR_{1,n}^{\E_d}$ or $\LL\phi, \cdots, \phi\RR_{1,n}^{W_d}$. Thus we have
$$\LL\alpha_1\psi_1^{\ell_1},\cdots,\alpha_k\psi_n^{\ell_n}\RR_{g,n}^{\clubsuit}\in \C\left[f_d, f_d', f_d'', \cdots\right]= \C\left[f_d, f_d', f_d''\right].$$
The last equality follows from \eqref{BP-relation}.
\end{proof}

\section{A genus-one FJRW invariant}
\label{secgenusoneinitialvalues}


Throughout this section, we consider the $d=3$ case, with
$W_3=x_1^3+x_2^3+x_3^3$ and $G=\mu_3$. We focus on the following genus-one FJRW invariant (see \eqref{FJRW-function}) with $n=3$

\begin{equation*}
\Theta_{1,n}:=\LD \underbrace{\phi, \cdots, \phi}_{n}\RD_{1,n}^{(W_3, \mu_3 )}.
\end{equation*}

Combining the computations in \cite{LLSZ}, we will prove
\begin{prop}\label{propinitialinvariant}
\cite[Theorem 1.1]{LLSZ}
\label{propinitialFJRWinvariant}
For the $(W_3, \mu_3 )$ case,  one has the following
FJRW invariant 
\begin{equation}
\label{invariant}
\Theta_{1,3}=\LD \phi,\phi,\phi\RD^{(W_3, \mu_3 )}_{1,3}={1\over 108}\,.
\end{equation}
\end{prop}

We first obtain a formula that express the Witten's top Chern class for $\Theta_{1,3}$ in terms of a Witten's top Chern class of three spin curves in Lemma \ref{cube}. 
Then  in Proposition \ref{prep1} and Corollary \ref{cor3-spin-class}, we analyze the later virtual class explicitly  by cosection localization. 
Finally, Proposition \ref{propinitialinvariant} will be deduced from these results and explicit computations in \cite{LLSZ}.

\subsection{Witten's top Chern class}
We begin with a formula for a Witten's top Chern class of the moduli of three-spin curves. 
The relevant moduli $\Mbar_{g=1,2^3}(W_3, \mu_3)$ (defined in \cite{CLL}) is the moduli of families 
\beq\label{xi}
\xi=[\Sigma\sub\sC,(\sL_i,\rho_i)_{i=1}^3]
\eeq
such that $\Sigma\sub\sC$ is a family of genus-one $3$-pointed twisted nodal curves,
each marking is a stacky point of automorphism group $\mu_3$,
$\rho_i:\sL_i^{\otimes 3}\cong\omega_\sC^{\log}$ are isomorphisms together with isomorphisms $\sL_i\cong \sL_1$ 
for $i=2$ and $3$ understood, the monodromy of $\sL_1$ along $\Sigma_i\sub\Sigma$ is 
$\frac{3-1}{3}$.
\footnote{Our convention is that
for $\sC=[\mathbb{A}^1/\mu_r]$ and an invertible sheaf of $\sO_\sC$-modules having monodromy 
$\frac{a}{r}\in [0,1)$ at $[0]$, then locally the sheaf takes the form $\sO_{\mathbb{A}^1}(a[0])/\mu_r$.}
Because of the isomorphisms $\sL_i\cong \sL_1$, we have canonical isomorphism
$$\cW_3\defeq \Mbar_{1,2^3}^{1/3}\cong 
\Mbar_{1,2^3}(W_3, \mu_3).
$$
where recall that $\cW_3$ parameterizes families of 
$\xi=[\Sigma\sub\sC,\sL,\rho]$ with objects $\Sigma$, $\sC$, $\sL$ and $\rho$ as before.


Let 
$$[\Mbar_{1,2^3}(W_3, \mu_3)^p]\virt\in A\lsta \Mbar_{1,2^3}(W_3, \mu_3)
$$ 
be the FJRW invariant of the pair $(W_3, \mu_3)$,
which is defined in \cite{CLL} as the cosection localized virtual cycles of the moduli stack
$\Mbar_{1,2^3}(W_3, \mu_3)^p$, parameterizing 
$$\xi=\{(\sC,\Sigma, \sL_1,\cdots,\varphi_1,\varphi_2,\varphi_3): (\sC,\Sigma, \sL_1,\cdots)\in \Mbar_{1,2^3}(W_3, \mu_3);
\varphi_i\in \Gamma(\sL_i)\}.
$$
As shown in \cite{CLL}, it has a cosection localized virtual cycle, denoted by $[\Mbar_{1,2^3}(W_3, \mu_3)^p]\virt$.
We let 
$$[\Mbar_{1,2^3}^{1/3,p}]\virt\in A\lsta \Mbar_{1,2^3}^{1/3}
$$ 
be the similarly defined its cosection localized virtual cycle. 

\begin{lem}\label{cube}
We have identity
\beq\label{W3}
[\Mbar_{1,2^3}(W_3, \mu_3)^p]\virt=\bl [\Mbar_{1,2^3}^{1/3,p}]\virt \br^3\in A^3\cW_3\equiv A_0\cW_3\,.
\eeq
\end{lem}

\begin{proof}
First we have the following Cartesian product
$$\begin{CD}
\Mbar_{1,2^3}(x^3, \mu_3)^p\times \Mbar_{1,2^3}(x^3,\mu_3)^p 
@<<< \Mbar_{1,2^3}(x^3+y^3,(\mu_3)^2)^p\\
@VVV @VVV\\ 
\Mbar_{1,2^3}(x^3,\mu_3)\times\Mbar_{1,2^3}(x^3,\mu_3) 
@<{f}<< \Mbar_{1,2^3}(x^3+y^3,(\mu_3)^2),\\
\end{CD}
$$
where the morphism $f$ is defined via sending $(\sC,\Sigma,\sL_1,\sL_2)$ to 
$$\bl(\sC,\Sigma,\sL_1),(\sC,\Sigma,\sL_2)\br.$$
Applying \cite[Thm 4.11]{CLL}, we get that 
\beq[\Mbar_{1,2^3}(x^3+y^3,(\mu_3)^2)^p]\virt= f\sta \bl [\Mbar_{1,2^3}(x^3,\mu_3)^p]\virt\times  [\Mbar_{1,2^3}(x^3,\mu_3)^p]\virt\br.
\label{fsta}\eeq
Now let 
$$g:  \Mbar_{1,2^3}(x^3+y^3,\mu_3)=\Mbar_{1,2^3}(x^3,\mu_3)\lra  \Mbar_{1,2^3}(x^3+y^3,(\mu_3)^2)
$$
be the diagonal morphism, then 
$$f\circ g: \Mbar_{1,2^3}(x^3,\mu_3)\lra\Mbar_{1,2^3}(x^3,\mu_3)\times\Mbar_{1,2^3}(x^3,\mu_3)
$$
is the diagonal morphism. As $g$ is \'etale and proper, we conclude
\beq\label{gsta}
[\Mbar_{1,2^3}(x^3+y^3,\mu_3)^p]\virt= g\sta [\Mbar_{1,2^3}(x^3+y^3,(\mu_3)^2)^p]\virt.
\eeq
Combined with \eqref{fsta} and \eqref{gsta}, we obtain
$$[\Mbar_{1,2^3}(x^3+y^3,\mu_3)^p]\virt=(f\circ g)\sta \bl [\Mbar_{1,2^3}(x^3,\mu_3)^p]\virt\times  [\Mbar_{1,2^3}(x^3,\mu_3)^p]\virt\br,
$$
which is $\bl[\Mbar_{1,2^3}(x^3,\mu_3)^p]\virt\br^2$. Here we have used that $\Mbar_{1,2^3}(x^3,\mu_3)$ is smooth.
Repeating the same argument, go from $x^3+y^3$ to $W_3$, we prove the lemma.
\end{proof}

\subsubsection{Cosection localized virtual cycles}

Let $\cW$ be a smooth DM stack, with a complex of locally free sheaves of $\sO_\cW$-modules
\begin{equation}\label{1}
\cE^\bullet:=[\sO_\cW(E_0)\mapright{s} \sO_\cW(E_1)]\,,
\end{equation}
of rank $a_0$ and $a_1=a_0+1$, respectively. 
Let $\pi: E_0\to \cW$ be the projection; the section $s$ induces a section $\ti s\in \Gamma(\ti E_1)$ of the pullback bundle 
$\ti E_1:=\pi\sta E_1$. We define
\beq\label{M}
\cM\defeq (\ti s=0)\sub E_0\,.
\eeq

\smallskip
\noindent
{\bf Assumption-I}. {\sl We assume $\cD=({\rm ker}\ s\neq 0)\sub \cW$ is a smooth Cartier divisor; $\text{Im}(s|_\cD)$ is a rank
$a_0-1$ subbundle of $E_1|_\cD$.
}
\smallskip

Because $\cD$ is a smooth Cartier divisor, we can find a vector bundle $F$ on $\cW$ fitting into
\begin{equation}\label{2}
\sO_\cW(E_0) \mapright{\eta_1} \sO_\cW(F) \mapright{\eta_2} \sO_\cW(E_1)
\end{equation}
so that $\eta_1|_{\cW-\cD}=s|_{\cW-\cD}$ is an isomorphism,
$F\to E_1$ is a subvector bundle, and $s=\eta_2\circ\eta_1.$

We let $\cA=H^1(\cE^\bullet)$. By Assumption-I, it fits into the exact sequence
\beq\label{01}0\lra \sO_\cW(E_0)\mapright{\phi}\sO_\cW(F)\lra\cA\lra 0.
\eeq
Further, there is a line bundle
$A$ on $\cD$ so that $\cA=\sO_\cD(A)$. 
In the following, we will view $c_1(A)\in A^1\cD$. Then for the inclusion $\iota:\cD\to\cW$, $\iota\lsta c_1(A)\in A^2\cW$.
Since $A$ is a line bundle on $\cD$, we have $c_1(\cA)=[\cD]$,
thus
\begin{lem}\label{lem1.1}
We have identity $c_1(E_1-F)=c_1(E_1- E_0)-[\cD]$.
\end{lem}

We let $J\sub E_0|_\cD$ be the kernel of $s|_\cD$; by our assumption it is a line bundle on $\cD$. We relate
$A$ to $J$.

\begin{lem}\label{J}
Let the situation be as stated, and assume Assumption-I, then $A\cong J(\cD)$.
\end{lem}

\begin{proof}
Let $\cJ=\sO_\cD(J)$ and let $\eta=\ker\{\sO_\cD(F)\to\cA\}$. Then $\eta$ fits into the exact sequences
$$0\to \sO_\cD(J)\to\sO_\cD(E_0)\to \eta\to 0\quad \text{and}\quad 0\to \eta\to\sO_\cD(F)\to \sO_\cD(A)\to 0.
$$
Let $\xi\in \sO_\cD(J)$ be any (local) section. Let $\ti\xi\in\sO_\cW(E_0)$ be a lift of the image of $\xi$ in
$\sO_\cD(E_0)$. Then $\phi(\ti\xi)\in \sO_\cW(F)$, where $\phi$ is as in \eqref{01}. Clearly, $\phi(\ti\xi)|_\cD=0$.
Let $t\in\sO_\cW(\cD)$ be the defining equation of $\cD$. Then $t\upmo \phi(\ti\xi)\in \sO_\cW(F)(-\cD)$. 
We define $\varphi(\xi)$ be the image
of $t\upmo\phi(\ti\xi)$ in $\sO_\cD(A(-\cD))$, under the composition 
$$\sO_\cW(F)(-\cD)\lra \sO_\cD(F(-\cD))\lra \sO_\cD(A(-\cD)).
$$

It is direct to check that $\varphi$ is a well-defined homomorphism of sheaves $\sO_\cD(J)\to \sO_\cD(A(-\cD))$,
and is an isomorphism. This proves the Lemma.
\end{proof}

This way, $\cM$ (cf. \eqref{M}) is a union of $\cW\sub E_0$ (the $0$-section) and the subbundle $J\sub E_0|_\cD\sub E_0$.
As $\cM\sub E_0$ is defined by the vanishing of $\ti s$, it comes with a normal cone
\beq\label{C}\bC\defeq\lim_{t\to 0}\Gamma_{t\upmo \ti s}\sub \ti E_1|_\cM.
\eeq

\begin{lem}\label{lem1.2} With Assumption-I,
the cone $\bC\sub\ti E_1|_\cM$ is a union of two subvector bundles $\eta_2(F)\sub E_1$ and 
$\pi\sta\eta_2(F)|_J\sub \ti E_1|_J$.
\end{lem}

\begin{proof}
This is local, thus without loss of generality we can assume $a_0=1$.
Since $\cD=(s=0)$ is a smooth divisor in $\cW$, near a point at $\cD$ we can give $\cW$ an analytic neighborhood $U$ with chart $(u,x)$, where $u$ is a multi-variable, so that $\cD=(x=0)$ and $s|_U:E_0|_U\to E_1|_U$ 
takes the form 
$$s|_U=(x,0): \sO_U\to \sO_U\oplus \sO_U^{\oplus (a_1-1)}\cong \sO_U(E_1).
$$
We let $y$ be the fiber-direction coordinate of $E_0|_U$. Then $\pi\upmo(U)\sub E_0$ has the chart 
$(u,x,y)$, with $\ti s|_{\pi^{-1}(U)}=(xy,0)$.
Therefore, the cone $\bC\sub E_0$ over $\pi\upmo(U)$ is the line bundle
$$\sO_{\pi\upmo(U)\cap \cM}\sub 
\sO_{\pi\upmo(U)\cap \cM}\oplus \sO_{\pi\upmo(U)\cap \cM}^{\oplus (a_1-1)}
\cong \sO_{\pi\upmo(U)\cap \cM}(\ti E_1).
$$
This proves the Lemma.
\end{proof}

\smallskip
\noindent
{\bf Assumption-II}. {\sl 
We assume that there is a homomorphism (cosection)
$$\sigma: \ti E_1|_\cM\lra\sO_\cM$$
so that $\sigma|_{\cW}=0$, and
$\pi\sta\eta_2(F)|_J$ lies in the kernel of $\sigma$.}

 Let 
$$[\cM]\virt_\sigma\defeq 0_\sigma^![\bC]\in A^{a_1-a_0} \cW
$$ 
be the image  of $[\bC]$ under the cosection localized Gysin map.

\begin{prop}\label{prep1}
Let the situation be as mentioned, and the cosection $\sigma$ is fiberwise homogeneous of degree $e$. 
Then 
$$[\cM]_\sigma\virt=-c_1(E_0-E_1)-(e+1)[\cD]\in A^1 \cW,\quad\text{when}\ a_1-a_0=1\,.$$
\end{prop}

\begin{proof}
Following the discussion leading to \cite[Lemma 6.4]{CL}, we compactify $\cM$ by compactifying
$J$ by $\cP\defeq \PP_\cD(J\oplus 1)$. Let $\cD_\infty=\PP_\cD(J\oplus 0)\sub
\PP_\cD(J\oplus 1)$. Then $\cP=J\cup \cD_\infty$, and
$\barcM= \cP\cup \cW$. Let $\bar\pi: \cP\to \cD$ be the tautological projection. 
Then $\pi\sta F|_{J}\sub \ti E_1|_{J}$ extends to $\bar\pi\sta F\sub \bar\pi\sta E_1$,
a subbundle.
Because $\sigma$ is fiberwise homogeneous of degree $e$, we see that 
$\sigma|_J: \ti E_1|_J=\bar\pi\sta E_1|_{J}\lra\sO_{J}$
extends to 
a homomorphism
$$\bar\sigma: \bar\pi\sta E_1(-e\cD_\infty)\lra\sO_{\cP},
$$
surjective along $\cD_\infty=\barcM-\cM$. 

We let $\bar\pi\sta F(-e\cD_\infty)\sub \bar\pi\sta E_1(-e\cD_\infty)$ be the associated twisting of the subbundle 
$\bar\pi\sta F\sub \bar\pi\sta E_1$. Applying \cite[Lemma 6.4]{CL}, we conclude that
\beq\label{00}0_\sigma^![\bC]=0_{E_1}^![F]+\bar\pi\lsta\bl 0^!_{\bar\pi\sta E_1(-e\cD_\infty)}[\bar\pi\sta F(-e\cD_\infty)]\br.
\eeq

When $a_1-a_0=1$, 
\begin{align*}
0^!_{\bar\pi\sta E_1(-e\cD_\infty)}[\bar\pi\sta F(-e\cD_\infty)]&=c_1\bl\bar\pi\sta (E_1/F)(-e\cD_\infty)\br
=\bar\pi\sta c_1(E_1/F)-e[\cD_\infty].
\end{align*}
Thus
$\bar\pi\lsta\bl 0^!_{\bar\pi\sta E_1(-2\cD_\infty)}[\bar\pi\sta F(-e\cD_\infty)]\br=-e[\cD]$.
Combined with Lemma \ref{lem1.1}, we prove the lemma.
\end{proof}

\subsection{Applying to FJRW invariant}

We let $\cM=\Mbar_{1,2^3}^{1/3,p}$. We claim that there is a complex of vector bundle as in \eqref{1} so that $\cM$ is defined as in
\eqref{M}, and there is a cosection $\sigma$ as in Assumption-II satisfying the condition stated.

Indeed, let $\Mbar_{1,2^3}$ be the moduli of 3-pointed genus one twisted curves with all markings are $\mu_3$ stacky. 
Then the forgetful morphism $q: \Mbar_{1,2^3}^{1/3}\to \Mbar_{1,2^3}$ is finite and smooth. 
Further, let $(\Sigma\sub \sC,\sL)$ be the universal family of 
$\Mbar_{1,2^3}^{1/3}$, then $(\Sigma\sub \sC)$ is the pull back of the universal family of $\Mbar_{1,2^3}$. 
Then a standard method shows that we can find a complex $\cE^\bullet=[s:\sO_{\sC}(E_0)\to \sO_{\sC}(E_1)]$ 
of locally free sheaves so that
$\cE^\bullet= R^{\bullet}\pi\lsta \sL$, in the derived category. Here $\pi: \sC\to\Mbar_{1,2^3}^{1/3}$ is the projection.
Then a standard argument shows that this complex $\cE^\bullet$ is the desired one, giving a canonical embedding of
$\cM=\Mbar_{1,2^3}^{1/3,p}$ into the total space of $E_0$, as the vanishing locus of $\tilde s$. 

The choice of cosection $\sigma$ is induced by $\sO_\cW(E_1)\to H^1(\cE^\bullet)$, following that in \cite{CLL}, and satisfies
Assumption-II.
Finally, following the construction of $[\Mbar_{1,2^3}^{1/3,p}]\virt$, we see that
$$[\cM]\virt_\sigma=[\Mbar_{1,2^3}^{1/3,p}]\virt.
$$
We skip the details here.

We next check that the Assumption-I holds in this case.

\begin{lem}
Let $\cD\sub \cW$ ($=\Mbar_{1,2^3}^{1/3}$) be the locus where $R^0\pi\lsta\sL$ is non-trivial, then it is a smooth divisor of $\cW$.
\end{lem}

\begin{proof}
Let $(\sC,\Sigma, \sL)\in \cW$ be a closed point so that $H^0(\sL)\ne 0$. Then a direct calculation shows that $\sC$ has a node
$q\in \sC$ that separates $\sC$ into two irreducible components $\sE$ and $\sR$, so that $q\sub \sE$ is a 1-pointed (twisted) elliptic curve
with $h^0(\sL|_\sE)=1$, and $q\cup \Sigma\sub\sR$ is a 4-pointed (twisted) rational curve. The same argument shows that
the converse is also true.
Therefore, letting $\cD\sub\Mbar_{1,2^3}^{1/3}$ be the closed locus (see Fig. \ref{divisor} below) where $R^0\pi\lsta\sL$ is non-trivial, $R^0\pi\lsta\sL$
is a locally free sheaf of $\sO_{\cD}$-modules. Equivalently, this says that, letting 
$$\pi_{\cD}:\sC_\cD=\sC\times_{\Mbar_{1,2^3}^{1/3}}\cD\lra \cD
$$
be the projection, then $\pi_{\cD\ast}\bl \sL|_{\sC_{\cD}}\br$ is a rank one locally free sheaf of $\sO_\cD$-modules. Let $t$ be a local
section of this sheaf, then $(t=0)\sub\sC_\cD$ becomes a family of rational curves, the family that contains all those $q\cup \Sigma\sub\sR$ 
mentioned. This shows that $\sC_\cD\to\cD$ is exactly the subfamily in $\Mbar_{1,2^3}^{1/3}$ that can be decomposed into 1-pointed
twisted elliptic curves $q\sub \sE$ with $h^0(\sL|_\sE)=1$, and 4-pointed twisted rational curves $q\cup \Sigma\sub \sR$. This 
implies that $\cD$ is a smooth divisor of $\cW=\Mbar_{1,2^3}^{1/3}$.
\end{proof}

We illustrate the divisor $\cD$ by a decorated graph in Figure \ref{divisor}
 below.
A generic point in $\cD$ consists a nodal curve with a genus one component (in blue) and a genus zero component (in green).
The monodromy along the node is ${1\over 3}$ on the genus-one component and ${2\over 3}$ on the genus-zero component.
Here $h^0=1$ is the rank of $R^0\pi\lsta\sL$ restricted on the genus-one component.
\begin{figure}[H]
\centering
\includegraphics[width=0.35\textwidth]{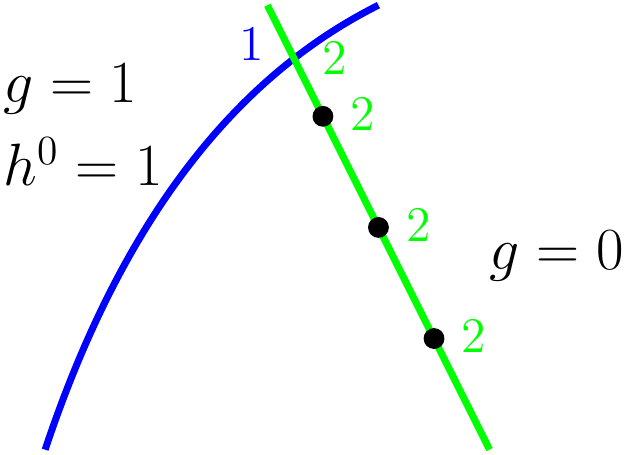}
\caption[divisor]{The divisor $\mathcal{D}$ on the moduli stack $\cW=\Mbar_{1,2^3}^{1/3}$.}
\label{divisor}
\end{figure}

Finally, to apply Proposition \ref{prep1}, we need to show that the cosection is fiberwise homogeneous of degree $e=2$. This
follows from the definition of the cosection in \cite{CLL}, and the degree $e$ is $3-1$, where $3$ is the denominator of $1/3$.
Applying Proposition \ref{prep1}, we obtain\footnote{This formula is a special case of a sequence of formulas for moduli of $r$-spin curves, conjectured by Janda \cite{Janda}.}
\begin{cor}\label{cor3-spin-class}
The Witten's top Chern class of the moduli of three-spin curves $\Mbar_{1,2^3}^{1/3}$ is 
\begin{equation}
\label{3-spin-class}
\bigl[\overline{\mathcal{M}}^{1/3,p}_{1,2^{3}}\bigr]\virt
=
-c_{1}(R^{\bullet}\pi_{*}\sL)-3[\cD] .
\end{equation}
\end{cor}



Applying Lemma \ref{cube}, we get
\begin{equation}
\label{three-spin-class}
\Theta_{1,3}=\deg [\Mbar_{1,2^3}(W_3,\mu_3)^{p}]\virt=\deg \bl\bigl[\overline{\mathcal{M}}^{1/3,p}_{1,2^{3}}\bigr]\virt\br^3.
\end{equation}
Thus the FJRW invariant $\Theta_{1,3}$ in Proposition \ref{propinitialinvariant}  can be calculated explicitly from the triple self-intersection of the cycle \eqref{3-spin-class}. 
Note that the first term in \eqref{3-spin-class} can be calculated by Chiodo's formula \cite{Chiodo: 2008}. 
The calculation is subtle and lengthy. The details  are given in \cite{LLSZ}.
An alternative approach in computing this invariant using the Mixed-Spin-P fields method developed in \cite{CLLL15, CLLL16} is also  presented in \cite{LLSZ}.

\section{LG/CY correspondence for the Fermat cubic}
\label{secLGCY}

This section is devoted to proving Theorem \ref{main-theorem}.
We shall show that the GW/FJRW correlation functions as Fourier/Taylor expansions of the same quasi-modular form
around different points (the infinity cusp and an interior point on the upper-half plane) which are related
by the so-called holomorphic Cayley transformation that we shall introduce.

\subsection{Cayley transformation and elliptic expansions of quasi-modular forms}
\label{secintroCayley}

It is well known that the Eisenstein series $E_2(\tau)$ is not modular, however its non-holomorphic modification
\begin{equation}
\label{e2-completion}
\widehat{E}_2(\tau, \bar\tau):=E_2(\tau)-{3\over \pi\, {\rm Im}(\tau)}
\end{equation}
is modular. The map (called modular completion) sending $E_2$ to $\widehat{E}_2$, and $E_4, E_6$ to themselves is
an isomorphism from $\widetilde{M}_*(\Gamma)$ to the \emph{ring of almost-holomorphic modular forms}
\begin{equation}
\widehat{M}_*(\Gamma):=\mathbb{C}[\widehat{E}_2, E_4, E_6]\,.
\end{equation}
More precisely, for any quasi-modular form $f(\tau)\in \widetilde{M}_*(\Gamma)$ of weight $k$,
we denote by $\widehat{f}(\tau,\bar\tau)\in \widehat{M}_*(\Gamma)$ its modular completion.
The function $\widehat{f}$ can be regarded as a polynomial in the
formal variable $1/ \mathrm{Im}\tau$
\begin{equation}\label{eqnmodularcompletion}
\widehat{f}=f+\sum_{j=1}^{k} f_{j}\cdot \left({1\over \mathrm{Im} \tau}\right)^{j},
\end{equation}
with coefficients some holomorphic functions $f_{j}$, $j=1,2,\cdots k$, in $\tau$.
We call the inverse of the modular completion the \emph{holomorphic limit}:
it maps the almost-holomorphic modular form $\widehat{f}$ in
 \eqref{eqnmodularcompletion}
to its degree zero term $f$ in the formal variable $1/ \mathrm{Im}\tau$.

For any point $\tau_*\in \mathbb{H}$, we form the Cayley transform from $\mathbb{H}$ to a
disk $\mathbb{D}$ (of appropriate radius
determined by $\tau_{*}$ and $c\neq 0$)
\begin{equation}
\label{eqnCayleytransform}
\tau \mapsto s(\tau):=c\cdot2\pi\sqrt{-1}(\tau_*-\bar{\tau}_*){\tau-\tau_*\over \tau-\bar{\tau}_*}\in\mathbb{D}\,.
\end{equation}
It is biholomorphic and we denote its inverse by $\tau(s)$.

Following \cite{Zagier:2008}, \cite{Shen:2016} defined a Cayley transformation $\mathscr{C}_{\tau_{*}}$
based on the action \eqref{eqnCayleytransform} on the space of almost-holomorphic modular forms:
it maps the almost-holomorphic modular form $\widehat{f}\in \widehat{M}_*(\Gamma)$ to  
\begin{equation}
\label{eqncayleytransformation}
\mathscr{C}_{\tau_{*}}(\widehat{f})(s, \overline{s}):=( 2\pi\sqrt{-1} c)^{-{k\over 2}} \cdot \left({\tau(s)-\bar\tau_{*} \over \tau_{*}-\bar\tau_{*} }\right)^k \cdot  \widehat{f}(\tau(s),\overline{\tau}(s))\,.
\end{equation}
This gives a natural way to expand an almost-holomorphic modular form near $\tau=\tau_*$. \\

A similar notion of holomorphic limit can be defined near the interior point $\tau_{*}$.
Computationally, this amounts to
taking the degree zero term in the $\bar{s}$-expansion of \eqref{eqncayleytransformation} (now regarded as a real-analytic function in $s,\bar{s}$)
using the structure \eqref{eqnmodularcompletion}.
This procedure induces a transformation $\mathscr{C}_{\tau_*}^{\rm hol}$ on quasi-modular forms.
The transformation $\mathscr{C}_{\tau_*}^{\rm hol}$ will be called the {\em holomorphic Cayley transformation} in the present work.
This transformation can be shown to respect the
differential ring isomorphism between the differential ring of quasi-modular forms and the differential ring of almost-holomorphic modular forms.
We illustrate the construction by the commutative diagram Figure  \ref{figureinvarianceRamanujan} below.
Interested readers are referred to \cite{Shen:2016} for details.

\begin{figure}[H]
  \renewcommand{\arraystretch}{1}
\begin{displaymath}
\xymatrixcolsep{4pc}\xymatrixrowsep{4pc}\xymatrix{  \widetilde{M}(\Gamma)
\ar@/^/[r]^{\textrm{modular completion}}\ar@{.>}[d]^{\mathscr{C}_{ \tau_*}^{\rm hol} }
 &  \ar@/^/[l]^{\textrm{constant term map}}  \widehat{M}(\Gamma) \ar[d]^{\mathscr{C}_{\tau_*}} 
 \\
 \mathscr{C}^{\rm hol}_{\tau_{*}}(\widetilde{M}(\Gamma)) \ar@/^/[r]^{\textrm{modular completion}}&  \ar@/^/[l]^{\textrm{holomorphic limit}}  \mathscr{C}_{\tau_*}(\widehat{M}(\Gamma))
 }
\end{displaymath}
 \caption[InvarianceofRamanujan]{Cayley transformation on quasi-modular and almost-holomorphic modular forms.}
  \label{figureinvarianceRamanujan}
\end{figure}
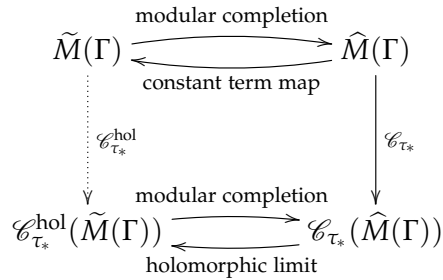

In this work we are mainly concerned with the expansions of the quasi-modular form $E_{2}$
around the infinity cusp $\sqrt{-1}\infty$ and the elliptic points
\begin{equation}
\label{eqnellipticpoint}
\tau_{*}=-{1\over 2\pi \sqrt{-1}}\cdot {\Gamma\left({1\over d}\right) \Gamma\left(1-{1\over d}\right)} e^{-{\pi\sqrt{-1}\over d}}\,,
\quad d\in \{3,4,6\}\,.
 \end{equation}
For the Fermat cubic polynomial case $d=3$,
in \eqref{eqnCayleytransform} we take
\begin{equation}
\label{eqnchoiceforc}
c={1\over 2\pi \sqrt{-1}} {\Gamma({1\over d}) \over \Gamma(1-{1\over d})^2} e^{ -{\pi \sqrt{-1}\over d}}\,.
\end{equation}
The choices in \eqref{eqnellipticpoint} and \eqref{eqnchoiceforc} then lead to the following rational expansion of $E_{2}$ around
$\tau_{*}$
\begin{equation}
\label{eqnellipticexpansionsofE2}
\mathscr{C}_{\tau_{*}}^{\mathrm{hol}}(E_{2})=-\frac{s^2}{ 9}-\frac{s^5}{1215}-\frac{s^8}{459270}+\cdots
\end{equation}
The other cases $d=4,6$ are similar.
All of these computations are easy following those in \cite{Shen:2016}.

\subsection{LG/CY correspondence}
\label{secexpansionofE2}

We consider the elliptic points \eqref{eqnellipticpoint} and the value \eqref{eqnchoiceforc}
for $c$ in \eqref{eqnCayleytransform}.
Theorem \ref{main-theorem} then follows from Theorem \ref{mainthm2} below.

\begin{thm}
\label{mainthm2}
Consider the LG space $([\mathbb{C}^3/\langle J_\delta\rangle], W)$ given by \eqref{weight-one} and \eqref{fermat-elliptic}, with $d=3$.
 \begin{enumerate}[(i)]
 \item
The genus-one GW correlation function is
\begin{equation}
\label{genus-1-gw}
-24\LL\omega\RR_{1,1}^{\E_{d}}(q)=E_2(q)\,.
\end{equation}

\item The GW correlation functions
$\LL\cdots\RR_{g,n}^{\E_{d}}$ are quasi-modular forms in the ring $\mathbb{C}[E_{2},E_{2}',E_{2}'']$.

\item
The genus-one FJRW correlation function $\LL\phi\RR^{W_{d}}_{1,1}(s)$ is
the Taylor expansion of $-{1\over 24}\cdot E_{2}$ around the elliptic point 
$$\tau_{*}=
-{\sqrt{-1}\over \sqrt{3}} \exp({2\pi \sqrt{-1}\over 3})
\in \mathbb{H}.$$
That is,
 \begin{equation}
\LL\phi\RR_{1,1}^{W_d}(s)=\mathscr{C}_{  \tau_{*}}^{\rm hol}\left(\LL\omega\RR_{1,1}^{\E_d}(q)\right)\,.
\end{equation}

 \item
The FJRW correlation functions $\LL\cdots\RR_{g,n}^{W_{d}}$
are holomorphic Cayley transformations of quasi-modular forms in the ring
$$\mathbb{C}[\mathscr{C}_{ \tau_{*}}^{\rm hol}(E_{2}),\mathscr{C}_{ \tau_{*}}^{\rm hol}(E_{2}'),\mathscr{C}_{ \tau_{*}}^{\rm hol}(E_{2}'')],$$ such that
 \begin{equation*}
 \mathscr{C}_{ \tau_* }^{\rm hol}\left(\LL\alpha_1\psi_1^{\ell_1}, \cdots, \alpha_n\psi_n^{\ell_n}\RR_{g,n}^{\E_{d}}(q)\right)
=\LL\Psi(\alpha_1)\psi_1^{\ell_1},\cdots, \Psi(\alpha_n)\psi_n^{\ell_n}\RR_{g,n}^{W_{d}}(s)\,.
 \end{equation*}
\end{enumerate}
\end{thm}

\begin{proof}
Part (i) is a well-known result in the literature, see e.g. \cite{Okounkov:2006}. We give a new proof based on the Chazy equation.
In order to get \eqref{genus-1-gw}, it suffices to check\footnote{Note that only two initial conditions are
needed to determine a solution from the space of formal power series in $q=e^{2\pi i \tau}$.}
 \begin{equation*}
 \LD\omega\RD_{1,1,0}^{\E_{d}}=-{1\over 24}\,, \quad \LD\omega\RD_{1,1,1}^{\E_{d}}=1\,.
  \end{equation*}
Both invariants can be obtained by analyzing the virtual fundamental classes explicitly.

\

Part (ii) is a consequence of Part (i), the Ramanujan identities \eqref{ramanujan}, and Proposition \ref{main-lemma2}.

\

For part (iii), the Selection rule \cite[Proposition 2.2.8]{Fan:2013} implies $\Theta_{1,1}=\Theta_{1,2}=0$ as the corresponding moduli spaces are empty.
On the other hand, according to
Proposition \ref{propinitialFJRWinvariant}, $$\Theta_{1,3}={1\over 108}.$$
Now we see that as a formal power series in $s$, the first three terms of $\LL \phi \RR_{1,1}^{W_d}(s)$ matches with those obtained from $\mathscr{C}_{\tau_{*}}^{\rm hol}(E_{2})$ in \eqref{eqnellipticexpansionsofE2}.
Since both $\LL \phi \RR_{1,1}^{W_d}(s)$ and $\mathscr{C}_{\tau_{*}}^{\rm hol}(E_{2})$ satisfies the Chazy equation \eqref{chazy}, 
we conclude that
\begin{equation*}
\LL \phi \RR_{1,1}^{W_d}(s)=-{1\over 24}\mathscr{C}_{\tau_{*}}^{\rm hol}(E_{2})\,.
\end{equation*}

\

For part (iv), we recall that by $g$-reduction, in either theory all non-trivial correlation functions
are differential polynomials in the building block $\LL\omega\RR_{1,1}^{\E_d}(q)$
or $\LL\phi\RR_{1,1}^{W_d}(s)$.
Since the holomorphic Cayley transformation respects the differential ring structure
and the $g$-reduction is independent of the CohFT in consideration, part (iv) is a consequence of part (iii), the Ramanujan identities \eqref{ramanujan}, and Proposition \ref{main-lemma2}.
\end{proof}
\begin{rem}

Proposition \ref{main-lemma1} and Proposition \ref{main-lemma2} hold for
all of the one-dimensional CY weight systems in  \eqref{weight-one}
 and  \eqref{fermat-elliptic}.
Provided the analogue of Proposition
\ref{propinitialFJRWinvariant}
for the $d=4$ or $6$ case is obtained, the same argument in the proof of
Theorem \ref{mainthm2}
generalizes straightforwardly.

\end{rem}

\section{Ancestor GW invariants for elliptic curves}
\label{secapplications1}

\label{application}

The tautological relations used in establishing Proposition \ref{main-lemma2} are not constructive and hence
not so useful for actual calculation of  higher-genus invariants.
For this reason, we make use of the
beautiful formulae for the descendent GW invariants of elliptic curves
given by Bloch-Okounkov \cite{Bloch:2000} reviewed below. 
For later use we also discuss the ancestor/descendent correspondence.

 \subsection{Higher-genus descendent GW invariants of elliptic curves}
 \label{sechighergenusdescendentGW}

In \cite{Okounkov:2006}, Okounkov and Pandharipande proved a correspondence between the stationary GW invariants and Hurwitz covers, called Gromov-Witten/Hurwitz correspondence.
To be more precise, let $\LD \prod_{i=1}^N \omega\wpsi_{i}^{\ell_i}\RD_{g,d}^{ \bullet \E}$ be the disconnected, {\em stationary}, descendent GW invariant of genus $g$ and degree $d$ (the number $N$ of markings is self-explanatory in the notation). Here $\wpsi_{i}$ is the descendent cotangent line class attached to the $i$th marking, and the symbol $\bullet$ stands for disconnected counting.
The invariant is called {\em stationary} as the insertions only involve the descendents of $\omega$.

Following \cite{Okounkov:2006}, we define the {\em $N$-point generating function}
\begin{equation}
\label{N-point-function}
F_N(z_{1},\cdots, z_{N}, q):=\sum_{\ell_1, \cdots, \ell_N \geq -2}\LL \prod_{i=1}^N \omega\wpsi_{i}^{\ell_i}\RR_{g}^{\bullet\E} \prod_{i=1}^{N} z_{i}^{\ell_i+1}\,,
\end{equation}
with the convention
\begin{equation*}
\LL\omega\wpsi^{-2}\RR_{0}^{\bullet\E}(q)=1\,.
\end{equation*}
The GW/Hurwitz correspondence \cite[Theorem 5]{Okounkov:2006} allows one to rewrite the  $N$-point generating function $F_N(z_{1},\cdots, z_{N}, q)$ by a beautiful character formula from \cite{Bloch:2000}
\begin{equation}\label{eqngeneratingseriesformula}
F_{N}(z_1,z_2,\cdots ,z_N,q)
=(q)_{\infty}^{-1}\cdot\sum_{\text{all permutations of} ~z_1,\cdots, z_{N}}  {\det M_{N}(z_{1},z_{2}\cdots, z_{N})\over   \Theta(z_{1}+z_{2}+\cdots+ z_{N})}\,,
\end{equation}
where  $M_{N}(z_{1},z_{2}\cdots, z_{N})$ is the matrix whose $(i,j), j\neq N$ entries are zero for $i>j+1$ and otherwise are given by
\begin{equation*}
{\Theta^{(j-i+1)}(z_{1}+\cdots+z_{N-j})  \over (j-i+1)! \Theta(z_{1}+\cdots+z_{N-j}) }\,,\,
j\neq N\,,
\quad
{\Theta^{(N-i+1)}(0)   \over (N-i+1)!}\,,\,
j=N\,.
\end{equation*}
Recall that $\Theta $ is defined to be the prime form
\begin{equation}\label{eqnprimeform}
\Theta(z)= {\vartheta_{({1\over 2}, {1\over 2} ) }(z,q)\over \partial_{z}\vartheta_{({1\over 2}, {1\over 2} ) }(z,q)|_{z=0}}=
2\pi \sqrt{-1} {\vartheta_{({1\over 2}, {1\over 2} ) }(z,q)\over -2\pi \eta^3}
=2\pi \sqrt{-1} e^{{1\over 24}E_{2}z^2}\sigma(z)\,.
\end{equation}
with
\begin{enumerate}
\item [(i)]
the Euler function 
$$(q)_{\infty}:=\prod_{n=1}^{\infty}(1-q^{n})$$ is related to the Dedekind eta function by
$\eta=q^{1\over 24}(q)_{\infty}$;
\item [(ii)]
 the Jacobi $\theta$-function
$$\vartheta_{({1\over 2}, {1\over 2} ) }(z,q):=\sum_{n\in \mathbb{Z}}q^{\pi \sqrt{-1} (n+{1\over 2} )^2}e^{(n+{1\over 2}) z}$$
has characteristic $({1\over 2}, {1\over 2} )$;
\item [(iii)]
the Weierstrass $\sigma$-function
 $\sigma(z)$ 
satisfies the following well-known formula\footnote{Note that the $z$-variable here differs from the usual one by a $2\pi \sqrt{-1}$ factor.} (see \cite{Silverman:2009arithmetic}),
\begin{equation}\label{eqnlogexpansionofprime}
\sigma(z)={z\over 2\pi  \sqrt{-1} }\exp\left(\sum_{k=2}^{\infty} {B_{2k}\over 2k (2k)!} z^{2k} E_{2k}\right)\,,
\end{equation}
where $B_{2k}, k\geq 1$ are Bernoulli numbers determined from
\begin{eqnarray*}
{x\over e^{x}-1}=1+{x\over 2}+\sum_{k=1}^{\infty}{B_{2k}\over (2k)!}x^{2k}\,.
\end{eqnarray*}

\end{enumerate}
Note that we often omit the subscript $g$ in the correlation function 
$$\LL \prod_{i=1}^N \omega\wpsi_{i}^{\ell_i}\RR_{g}^{\bullet \E}$$
which can be read off from the degree of the insertion according to the dimension axiom, we shall also omit the
 argument $q$ in the functions for ease of notation.\\

The formula \eqref{eqngeneratingseriesformula} provides an effective algorithm in computing the
stationary descendent GW invariants.
For example,
as already computed in \cite{Bloch:2000}, one has
\begin{equation}\label{eqnfirstfewNpointfunctions}
\begin{aligned}
F_{1}(z_1)&={1\over (q)_{\infty}\Theta(z_1)}\,,\quad\\
F_{2}(z_1,z_2)&={1\over  (q)_{\infty}\Theta(z_1+z_2)} (\partial_{z_{1}} \ln\Theta(z_1) +\partial_{z_{2}} \log \Theta(z_2))\,,\quad\\
&\cdots
\end{aligned}
\end{equation}

\begin{rem}\label{remArtinstack}
Let $ \LL \omega\RR^{\circ \E }$ be the generating series of stable maps with
connected domains with no descendent nor ancestor classes. Then one has the well known formula
\begin{equation}\label{eqnconnectedgenusonecorrelationfunction}
\LL \omega\RR^{\circ \E}= -{1\over 24} E_{2}\,.
\end{equation}
It is easy to see that
\begin{equation}
\LL \omega\RR^{\bullet \E}=
\LL \omega\RR^{\circ \E}\cdot \exp(G(q))\,,
\quad
G=\sum_{d\geq 1}\LL \RR^{\circ \E}_{g=1,d } q^{d}\,.
\end{equation}
One can show in this case by enumerating stable maps with
connected domains that
\begin{equation}
q{d\over dq} G=\sum_{d\geq 1}\LL \omega \RR^{\circ\E}_{g=1,d } q^{d}
=-q{d\over dq}\log (q)_{\infty} \,.
\end{equation}
Solving this equation and using the initial terms of $G$ which can be easily computed, one obtains
\begin{equation}\label{eqnartinstack}
G=-\log (q)_{\infty}\,.
\end{equation}
This then gives
\begin{equation}
\LL \omega\RR^{\bullet\E}= (q)_{\infty}^{-1}\cdot
\LL \omega\RR^{\circ\E}
=(q)_{\infty}^{-1} \cdot -{1\over 24}E_{2}\,.
\end{equation}
More generally, for the one-point GW correlation function, the same reasoning implies that
\begin{equation*}
\LL \omega \wpsi^{k}\RR^{\bullet\E}=(q)_{\infty}^{-1}\cdot
\LL \omega \wpsi^{k}\RR^{\circ\E}\,.
\end{equation*}

The result \eqref{eqnartinstack}, indicates that one can add an extra contribution from the degree zero
part to $G$, whose corresponding moduli is an Artin stack. This contribution can be defined to be $\log q^{-{1\over 24}}$.
In this way, after applying the divisor equation,
it yields the contribution $-1/24$ for the degree zero part in $\LL \omega \RR^{\circ\E} $.
This definition of the extra contribution for the Artin stack changes $(q)_{\infty}$ to $\eta$.
What one gains from the inclusion of this is the quasi-modularity of the GW generating functions.
The discrepancy will be further discussed from the viewpoint of ancestor/descendent correspondence below.
\end{rem}

It is shown in \cite{Bloch:2000} by manipulating the series expansions that
the descendent GW correlation functions are essentially (modulo the issue discussed in Remark \ref{remArtinstack}) quasi-modular forms.
By induction, the weight of $(q)_{\infty}\cdot\LL \prod_{i=1}^{N} \omega \wpsi_{i}^{k_i}\RR^{\bullet\E}$
is $\sum (k_{i}+2)$.
This can also be seen easily by using
\eqref{eqnprimeform}, \eqref{eqnlogexpansionofprime}.

\subsection{Ancestor/descendent correspondence}

Since explicit formulae in \cite{Bloch:2000} are available only for descendent GW invariants while
we are mainly concerned with  ancestor GW invariants, we shall first exhibit the relation between these two types of GW invariants.
The relation between the descendent GW invariants and the ancestor GW invariants are described for general targets in \cite[Theorem 1.1]{Kontsevich:1998}.
This is the so-called {\em ancestor/descendent correspondence}.
This correspondence is written down elegantly using a quantization formula of quadratic Hamiltonians in \cite[Theorem 5.1]{Giv01}.

We summarize some basics of quantization of quadratic Hamiltonians from \cite{Giv01}.
Let $H$ be a vector space of finite rank, equipped with a non-degenerating pairing $\LD-, -\RD$.
Let  $H(\!(z)\!)$ be the loop space of the vector space $H$, equipped with a symplectic form $\Omega$
\begin{equation*}
\Omega \Big(f(z), g(z)\Big):={\rm Res}_{z=0} \LD f(-z), g(z)\RD\,.
\end{equation*}

Let $\bt_k$ be the collection of variables $\bt_k=\{t^\alpha_{k}\}_{\alpha}$ where $\alpha$ runs over a basis of $H$, and $\bt$ be the collection
\begin{equation*}
\bt=\{\bt_0, \bt_1, \cdots\}\,.
\end{equation*}
We organize the collection $\bt_{k}$
into a formal series $\bt_k$
\begin{equation*}
\bt_k(z)=\sum_{i}t^i_k\cdot \alpha_i\cdot z^k\,.
\end{equation*}
Similar notations are used for $\bf{s}_k, \bf{s}$ 
below.
Introduce the {\em dilaton shift}
\begin{equation}
\label{dilaton-shift} 
\mathbf{q} (z)=\bt(z) -z\cdot\one\,.
\end{equation}

We consider an upper-triangular symplectic operator on $H(\!(z)\!)$, defined by 
\begin{equation*}
S(z^{-1}): =1+\sum_{i=1}^{\infty} z^{-i}S_i, \quad S_i\in {\rm End}(H)\,.
\end{equation*}
Given an element $\mathcal{G}(\mathbf{q})$ in certain Fock space, the quantization operator $\widehat{S}$ of a symplectic operatos $S$ gives another Fock space element
\begin{equation}
\label{quant-operator}
(\widehat{S}^{-1}\mathcal{G})(\mathbf{q})=e^{W(\mathbf{q},\mathbf{q})/2\hbar}\mathcal{G}([S \mathbf{q}]_+)\,,
\end{equation}
where
$[S\mathbf{q}]_+$ is the power series truncation of the function $S(z^{-1})\mathbf{q}(z)$, and the quadratic form
$W = \sum (W_{k\ell}{\mathbf{q}}_k, {\mathbf{q}}_\ell)$ is defined by
\begin{equation*}
\sum_{k, \ell\geq0}{W_{k, \ell}\over w^kz^\ell}:={S^*(w^{-1})S(z^{-1})-{\rm Id}\over w^{-1}+z^{-1}}\,.
\end{equation*}
Here ${\rm Id}$ is the identity operator on $H(\!(z)\!)$ and $S^*$ is the adjoint operator of $S$.

Following Givental \cite[Section 5]{Giv01}, for the descendent theory
we define a particular symplectic operator $S_{t}$ by
\begin{equation}
\label{s-operator}
(a, S_{t} b):=\langle a, \frac{b}{z-\psi}\rangle=:(a, b)+\sum_{k=0}^{\infty}\LL a, b\wpsi^k\RR_{0,2}^{\circ \E}\, z^{-1-k}\,.
\end{equation}

Now we specialize to the elliptic curve case and write down the quantization formula for the ancestor/descendent correspondence explicitly.
Henceforward, we use the following convention.
\begin{itemize}
\item Recall $\{\one, \mathfrak{b}_1, \mathfrak{b}_2 ,\phi\}$ is a basis of the FJRW state space $\mathscr{H}_{(W_d,G_d)}$ given in \eqref{cubic-basis}.
We parametrize the ancestor classes $\one\psi^{\ell},  \mathfrak{b}_1\psi^{\ell}, \mathfrak{b}_2\psi^{\ell},\phi\psi^{\ell}$
by
\begin{equation}
\label{fjrw-coordinates}
s^0_{\ell}, s^1_{\ell}, s^2_{\ell}, s^3_{\ell}\,.
\end{equation}
\item Recall $\{\one, e_1, e_2,\omega\}$ is a basis of the cohomology space $H^*(\E)$.
We parametrize the ancestor classes $\one\psi^{\ell}, e_1\psi^{\ell},e_2\psi^{\ell}, \omega\psi^{\ell}$ and
descendent classes $\one\wpsi^{\ell}, e_1\wpsi^{\ell}, e_2\wpsi^{\ell},\omega\wpsi^{\ell}$
by
\begin{equation}
\label{descedent-coordinates}
t^0_{\ell}, t^1_{\ell}, t^2_{\ell}, t^3_{\ell}\,;
\quad
\tilde{t}^0_{\ell}, \tilde{t}^1_{\ell}, \tilde{t}^2_{\ell}, \tilde{t}^3_{\ell}\,
\end{equation}
respectively.
\end{itemize}
The total descendent potential of the GW theory of $\E$ is defined by
\begin{equation}
\label{descedent-curve}
\cD^\E(\tilde{\bf t}):=\exp\left(\sum_{g\geq0}\hbar^{g-1}\cF^{\circ\E}_{g}(\tilde{\bf t})\right)
:=\exp\left(\sum_{g\geq0}\hbar^{g-1}\sum_{n\geq 0}\langle\tilde{\bf t}, \cdots, \tilde{\bf t} \rangle^{\circ\E}_{g,n}\right)\,.
\end{equation}
The total ancestor potential of the GW theory of $\E$ is defined by
\begin{equation*}
\cA^\E(\bt):=\exp\left(\sum_{g\geq0}\hbar^{g-1}\cF^{\circ\E}_{g}({\bf t})\right)
:=\exp\left(\sum_{g\geq0}\hbar^{g-1}\sum_{\substack{n\geq 0\\ 2g-2+n>0}}\langle{\bf t}, \cdots, {\bf t} \rangle^{\circ\E}_{g,n}\right)\,.
\end{equation*}
The total ancestor FJRW potential is defined similarly.

The quantity $\cF_1^{\circ \E}(t)$ is the genus-one primary potential of the GW theory of $\E$ appearing in $\mathcal{A}^{\E}$, with $q=e^{t}$ the parameter keeping track of the degree.
According to \cite[Theorem 5.1]{Giv01}, the Ancestor/descendent correspondence of the elliptic curve is given by
\begin{equation}
\label{ancestor-descendent}
\cD^{\E}=e^{\cF^{\circ \E}_{1}(t)}\widehat{S}_{t}^{-1} \cA^{\cE}\,,
\end{equation}
under the identification
$\tilde{t}^i_\ell=t^i_\ell$\,.

According to \eqref{eqnartinstack}, the genus-one potential is
\begin{equation*}
\cF_1^{\circ \E}(t)=G(q)=\sum_{d\geq 1}\LD\,\RD^{\circ \E}_{1,0,d}q^{d}=-\log (q)_{\infty}\,, \quad q=e^{t}\,.
\end{equation*}
Thus we obtain
\begin{equation*}
\widehat{S}_{t}^{-1} \cA^{\E}=e^{-\cF_1^{\circ \E}(t)}\cD^{\E}=(q)_{\infty}\cdot\cD^{\E}=(q)_{\infty}\cdot\sum_{g,n\in\mathbb{Z}}\hbar^{g-1}\LL \tilde{\mathbf{t}}, \cdots, \tilde{\mathbf{t}}\RR_{g,n}^{\bullet \E}\,.
\end{equation*}
A direct calculation of \eqref{s-operator} shows the restriction of $S_{t}$ on the odd cohomology is the identity operator, and the restriction to even cohomology is given by
\begin{equation*}
S_{t}
\begin{pmatrix}
\one\\
\omega
\end{pmatrix}
=
\begin{pmatrix}
1&{t\over z}\\
&1
\end{pmatrix}
\begin{pmatrix}
\one\\
\omega
\end{pmatrix}\,.
\end{equation*}
Now we write down an explicit formula for the quantization operator \eqref{quant-operator}.
The symplectic operator $S_{t}$ is given in terms of infinitesimal symplectic operator ${h(t)\over z}$,
\begin{equation*}
S_{t}=\exp\left({h(t)\over z}\right)\,,
\end{equation*}
where $h(t)\in {\rm End}(H)$ such that 
$h(t)(\one)=t\omega, h(t)(\omega)=0$, and $h(t)(e_i)=0$ otherwise.
In terms of the Darboux coordinates $\tilde{q}_{k}^{i}, \tilde{p}^{i}_{k}$,
the corresponding quadratic Hamiltonian has the form (see \cite[Section 3]{Lee08} for example)
\begin{equation*}
P\left({h(t)\over z}\right)=-t{(\tilde{q}_0^0)^2\over 2}-t\cdot\sum_{k\geq 0} \tilde{q}^0_{k+1} \tilde{p}^{0}_k\,.
\end{equation*}
Applying the quantization formula, we get
\begin{equation}
\label{quant-s-operator}
\widehat{S}_{t}=\exp\left(\widehat{P\left({h(t)\over z}\right)}\right)
=\exp\left(-t{(\tilde{q}_0^0)^2\over 2}-t\cdot\sum_{k\geq 0} \tilde{q}^0_{k+1}{\partial\over\partial \tilde{ q}^0_k}\right)\,.
\end{equation}
As a consequence, we observe that this operator has no influence on the parameter $\tilde{q}_{k}^3$ for the descendent $\omega \tilde{\psi}^k$. Thus we obtain
\begin{prop}\label{lem-ancestor-descedent-stationary}
The relation between the stationary descendent invariants and the corresponding ancestor invariants is given by
\begin{equation}
\label{ancestor-descedent-stationary}
(q)_{\infty}\cdot \LL \prod_{i=1}^N \omega\wpsi_{i}^{\ell_i}\RR_{g}^{ \bullet \E}=\LL \prod_{i=1}^N \omega\psi_{i}^{\ell_i}\RR_{g}^{ \bullet \E}\,.
\end{equation}
\end{prop}

Quasi-modularity for the correlation functions in the disconnected theory is equivalent to
those for the connected theory, as one can see by examining the generating series.
Hence our
Theorem \ref{mainthm2}(ii)
is consistent with the results in  \cite{Bloch:2000, Okounkov:2006} about the quasi-modularity via the above proposition.

\section{Higher-genus FJRW invariants for the Fermat cubic}
\label{secapplications2}

In this section, we give several applications of Theorem \ref{main-theorem}.
With the help of the Bloch-Okounkov formula \cite{Bloch:2000}, Cayley transformation allows us to compute the FJRW invariants of the Fermat elliptic polynomials at all genera.
It also transforms various structures for the GW theory of elliptic curves, such as the holomorphic anomaly equations \cite{Okounkov:2006, OP18}  and Virasoro constraints \cite{Okounkov-vir}, to those in the corresponding FJRW theory.

\subsection{Higher-genus ancestor FJRW invariants for the cubic}

Consider the Laurent expansion of the $N$-point generating function $ F_{N}(z_{1}, z_{2},\cdots, z_{N}, q)$.
The Laurent expansion of $\partial^{m}\ln \Theta$ is clear from \eqref{eqnlogexpansionofprime}, while
that of $1/\Theta$ or $1/\sigma$
can be obtained by applying the Fa\'a de Bruno formula to the exponential term in $1/\sigma$ which in the current case is determined by the Bell polynomials
in  $-B_{2k}E_{2k}/2k, k\geq 2$.
However, this only gives the Laurent coefficients in terms of the generators
$E_{2k}, k\geq 2$ for the ring of modular forms. The expansions obtained are not particularly useful for our later purpose which prefers a finite set of
generators only.

We proceed as follows.
First
the Taylor expansion of the Weierstrass $\sigma$-function is given by the classical result  \cite{Wer:1894}
\begin{equation}\label{eqnsigmaexpansionintermsofE4E6}
\sigma=\sum_{m,n\geq 0} {a_{m,n}\over (4m+6n+1)!} \left({2\pi^4\over 3} E_{4}\right)^{m} \left({16\pi^6\over 27} E_{6}\right)^{n} \left({z\over 2\pi  \sqrt{-1}} \right)^{4m+6n+1}\,,
\end{equation}
where the coefficients $a_{m,n}$ are complex numbers determined from
the Weierstrass recursion
\begin{align*}
a_{m,n}=&3(m+1)a_{m+1,n-1}+{16\over 3} (n+1)a_{m-2,n+1}\\
&-{1\over  6} (4m+6n-1) (4m+6n-2) a_{m-1,n}\,,
\end{align*}
with the initial values  $a_{0,0}=1$ and
 $a_{m,n}=0$ if either of $m,n$ is strictly negative.
The Laurent expansion of $1/\sigma$ is then obtained from the above.
It takes the form
\begin{equation}
{1\over \sigma}=  \sum_{m,n\geq 0} b_{m,n} \left({2\pi^4\over 3} E_{4}\right)^{m}  \left({16\pi^6\over 27} E_{6}\right)^{n}  \left({z\over 2\pi  \sqrt{-1} }\right)^{4m+6n-1}\,,
\end{equation}
for some $b_{m,n}$ that can also be obtained recursively.
The formula in
\eqref{eqnsigmaexpansionintermsofE4E6} also
 gives rise to the Laurent expansion of $\partial \ln \sigma$ and hence of
$\partial \ln \Theta$ in terms of the generators $E_{2}, E_{4}, E_{6}$.
Together with that of $\partial \ln \Theta$
it can be used to  compute the Laurent expansion of $F_{N}(z_{1},z_{2},\cdots, z_{N},q)$.\\

Consider the $N=1$ case first. According to  
\eqref{eqnfirstfewNpointfunctions}
the Laurent expansion of $F_{1}$ is given by
\begin{eqnarray*}
F_{1}(z,q)&=&{1\over 2\pi  \sqrt{-1}\cdot (q)_{\infty}} e^{-{1\over 24}E_{2}z^2} \sigma^{-1}\\
&=&{1\over z\cdot (q)_{\infty}}\sum_{\ell,m,n\geq 0} {b_{m,n}\over \ell!} \left(-{E_{2}\over 24}\right)^{\ell}     \left({E_4\over 24}\right)^{m} \left( - {E_6\over 108}\right)^{n} {z^{2\ell+4m+6n}}\,.
\end{eqnarray*}
We therefore arrive at the following relation for the descendent GW correlation functions
\begin{equation}\label{eqnonepointGWcorrelationfunctions}
(q)_{\infty}\cdot \LL \omega\wpsi^{k}\RR^{\bullet\E}= \sum_{\substack{\ell, m,n\geq 0\\
2\ell+4m+6n=k+2}}  {b_{m,n}\over \ell!} \left(-{E_{2}\over 24}\right)^{\ell}  \left({E_4\over 24}\right)^{m} \left( - {E_6\over 108}\right)^{n}
 \,,
 \quad k\geq -2\,.
\end{equation}
As explained in Proposition \ref{lem-ancestor-descedent-stationary},
this is the corresponding ancestor GW correlation function and is
indeed a quasi-modular form of weight $k+2$.
The first few Laurent coefficients are
\begin{equation}
1\,,\quad -{1\over 24} E_{2}\,,
\quad
{1\over 2^6 3^2}\left(
{1\over 5}E_{4}+{1\over 2}E_{2}^2\right)\,,
\quad \cdots
\end{equation}
The other cases are similar. For example, for the $N=2$ case from \eqref{eqnfirstfewNpointfunctions}
we write
\begin{equation*}
(q)_{\infty }\cdot F_{2}(z_1,z_2)={z_1+z_2\over \Theta(z_1+z_2) }\cdot { \partial_{z_{1}} \ln\Theta(z_1) +\partial_{z_{2}} \log \Theta(z_2) \over z_1+z_2}\,.
\end{equation*}
The first term on the right hand side is expanded as in the $N=1$ case, while the second term using \eqref{eqnprimeform} and \eqref{eqnlogexpansionofprime}.\\


Recall that the derivative on the level of generating series corresponds to the divisor equation in GW theory, and that taking derivatives commute with Cayley transformations as shown in \cite{Shen:2016}.
The generators of the differential ring of quasi-modular forms are $E_2, E_4, E_6$.
To deal with the differential structure, it is in fact more convenient to use the generators $E_2,E_2',E_2''$ for the ring of quasi-modular forms as opposed to $E_2, E_4, E_6$. 
By Therorem \ref{mainthm2}, the ancestor GW correlation functions satisfy 
\begin{equation}
 \LL   \prod_{i=1}^{N} \omega \psi_i^{k_{i}} \RR^{ \circ \E}\in \mathbb{C}[E_2,E_2',E_2'']\,.
\end{equation}
Theorem \ref{main-theorem} applies to the disconnected invariants (by examining the relation between the generating series)  and we have
\begin{equation}
\LL  \phi\psi_1^{k_1}, \dots, \phi\psi_N^{k_N}\RR_{g}^{ \bullet W_d}=\mathscr{C}_{ \tau_{*}}^{\mathrm{hol}}\left(\LL \omega\psi_1^{k_1}, \dots, \omega\psi_N^{k_N}\RR_{g}^{\bullet \E_d}\right)\,.
\end{equation}

Now we can apply Cayley the transformation directly to the disconnected, ancestor GW correlation functions and obtain the disconnected, ancestor FJRW correlation functions.
As computed in \eqref{eqnellipticexpansionsofE2}, for the $d=3$ case we have
 \begin{equation}
\mathscr{C}_{ \tau_{*}}^{\mathrm{hol}}(E_2)=
-\frac{s^2}{9}-\frac{s^5}{ 1215}-\frac{s^8}{  459270}
+\cdots
\end{equation}
Since $\mathscr{C}_{\tau_{*}}^{\mathrm{hol}}$ respects the product and the differential structure  \cite{Shen:2016},
the differential equations \eqref{ramanujan}
 imply
 \begin{equation}
\begin{dcases}
\mathscr{C}_{ \tau_{*}}^{\mathrm{hol}}(E_4)=\mathscr{C}_{ \tau_{*}}^{\mathrm{hol}}(E_2^2-12 E_2')
=
\frac{8s}{3}+\frac{5s^4}{ 81}+\frac{2s^7}{5103}+\cdots
\\
\mathscr{C}_{ \tau_{*}}^{\mathrm{hol}}(E_6)=\mathscr{C}_{ \tau_{*}}^{\mathrm{hol}}(E_2E_4-3 E_4')
=
-8-\frac{28s^3}{27}-\frac{7s^6}{405}+\cdots
\end{dcases}
\end{equation}
From \eqref{eqnonepointGWcorrelationfunctions}, Proposition \ref{lem-ancestor-descedent-stationary},  
Theorem \ref{main-theorem} and the degree formula \eqref{gw-degree}, we immediately obtain 
$$\LL  \phi\psi_1^{2g-2}\RR_{g,1}^{\bullet W_3}= \sum_{\substack{\ell, m,n\geq 0\\
\ell+2m+3n=g}}
{ b_{m,n}\over \ell!} \left(-{\mathscr{C}_{ \tau_{*}}^{\mathrm{hol}}(E_2)\over 24} \right)^{\ell}   \left({ \mathscr{C}_{ \tau_{*}}^{\mathrm{hol}}(E_4)\over 24}\right)^{m} \left( - { \mathscr{C}_{ \tau_{*}}^{\mathrm{hol}}(E_6)\over 108}\right)^{n}
 \,.
 $$
Now Corollary \ref{coronepointFJRW} follows from the fact that the disconnect and connected one-point ancestor functions are the same.

\subsection{Holomorphic anomaly equations}
We now describe holomorphic anomaly equations for the FJRW correlation functions.
In the rest of the paper we shall only discuss connected invariants and hence omit the supscript ``$\circ$" from the notations.


\subsubsection{HAE for ancestor GW correlation functions}
In \cite{OP18}, Oberdieck and Pixton use the polynomiality of double ramification cycles to prove that the GW cycles
$\Lambda_{g,n}^{\E}(\alpha_1,\cdots,\alpha_n)$
of the elliptic curves are cycle-valued quasi-modular forms. Take the derivative of those cycles with respect to the second Eisenstein series $E_2(q)$, they obtain a holomorphic anomaly equation \cite[Theorem 3]{OP18}.
As a consequence, intersecting the corresponding GW cycles with $\prod_{k}\psi_k^{\ell_k}$ on $\overline{\cM}_{g,n}$ leads to a holomorphic anomaly equation for the ancestor GW functions
\begin{equation*}
\LL\alpha_1\psi_1^{\ell_1}, \cdots, \alpha_n\psi_n^{\ell_n}\RR_{g,n}^{\E}(q)\in \mathbb{C}[E_2,E_4,E_6]\,.
\end{equation*}
For each subset $I\subseteq \{1, \cdots, n\}$, we use the following convention
\begin{equation*}
\alpha_{I}:=\{\alpha_i\psi_i^{\ell_i}, \forall i\in I\}\,.
\end{equation*}
For convenience we introduce the normalized Eisenstein series
\begin{equation*}
C_2(q)=-{1\over 24}\,E_2(q)\,.
\end{equation*}
It is a classical fact that the Eisenstein series $E_2, E_4,E_6$ are algebraically independent.
One has \cite{OP18} for the ancestor GW correlation functions 
\begin{align}
&{\partial\over \partial C_2}\LL\alpha_1\psi_1^{\ell_1}, \cdots, \alpha_n\psi_n^{\ell_n}\RR_{g,n}^{\E}(q)\label{hae-op-anomaly}\\
=&
\LL\alpha_1\psi_1^{\ell_1}, \cdots, \alpha_n\psi_n^{\ell_n}, \one, \one\RR_{g-1,n+2}^{\E}(q)+
\sum_{\substack{g_1+g_2=g\\ \{1, \cdots, n\}=I_1\coprod I_2}}\LL \alpha_{I_1}, \one\RR_{g_1}^{\E}(q)\LL \one, \alpha_{I_2}\RR_{g_2}^{\E}(q)\nonumber\\
&-2\sum_{i=1}^{n}\left(\int_\E \alpha_i\right)\LL\alpha_1\psi_1^{\ell_1}, \cdots, \one\psi_i^{\ell_i+1},\cdots,\alpha_n\psi_n^{\ell_n}\RR_{g,n}^{\E}(q)\nonumber
\end{align}


\begin{rem}
This equation can also be proved using only the combinatorial results reviewed in Section \ref{sechighergenusdescendentGW}, see Pixton \cite{Pixton:2008}.
\end{rem}

\subsubsection{HAE for ancestor FJRW correlation functions}

Recall that the holomorphic Cayley transformation $\mathscr{C}_{ \tau_* }^{\mathrm{hol}}$
respect the differential ring structure of the set of  quasi-modular forms.
Applying the holomorphic Cayley transformation to \eqref{hae-op-anomaly}, using
Theorem \ref{mainthm2}
we immediately obtain
the following HAE for the ancestor FJRW correlation functions.

\begin{cor}
\label{hae-lg}
Let the notations be as in Theorem \ref{main-theorem}.
For the $d=3$ case, the ancestor FJRW correlation function
\begin{equation*}
\LL\alpha_1\psi_1^{\ell_1}, \cdots, \alpha_n\psi_n^{\ell_n}\RR_{g,n}^{W_d}
\in \mathbb{C}[ \mathscr{C}_{ \tau_* }^{\mathrm{hol}}(C_2), \mathscr{C}_{ \tau_* }^{\mathrm{hol}}(E_4), \mathscr{C}_{ \tau_* }^{\mathrm{hol}}(E_6) ]\,,
\quad
C_{2}=-{1\over 24}E_{2}
\end{equation*}
satisfies
\begin{align*}
&{\partial\over \partial \mathscr{C}_{ \tau_* }^{\mathrm{hol}}(C_2)}\LL\alpha_1\psi_1^{\ell_1}, \cdots, \alpha_n\psi_n^{\ell_n}\RR_{g,n}^{W_d}\\
=&\LL\alpha_1\psi_1^{\ell_1}, \cdots, \alpha_n\psi_n^{\ell_n}, \one, \one\RR_{g-1,n+2}^{W_d}+
\sum_{\substack{g_1+g_2=g\\ \{1, \cdots, n\}=I_1\coprod I_2}}\LL \alpha_{I_1}, \one\RR_{g_1}^{W_d}\LL \one, \alpha_{I_2}\RR_{g_2}^{W_d}\\
&-
2\sum_{i=1}^{n}\LL\alpha_1\psi_1^{\ell_1}, \cdots, \delta_{\alpha_i}^{\phi}\one\psi_i^{\ell_i+1},\cdots,\alpha_n\psi_n^{\ell_n}\RR_{g,n}^{W_d}\,,
\end{align*}
where $\delta_{\alpha_i}^{\phi}$ is the Kronecker symbol.
\end{cor}

\subsection{Virasoro constraints}
Virasoro operators in Gromov-Witten theory were proposed by Eguchi, Hori, and Xiong \cite{Eguchi:1997} for Fano manifolds and later generalized to more general targets \cite{Dubrovin:1999, Giv01}.
The famous {\em Virasoro Conjecture} predicts that the total descendent potentials in GW theory are annihilated by the Virasoro operators. 
It is one of the most fascinating conjectures in GW theory. Despite significant developments in the literature, it remains open for a large category of targets.

The Virasoro conjecture for nonsingular target curves is solved by Okounkov and Pandharipande \cite{Okounkov-vir}.
In particular, 
when the target is an elliptic curve, the formulas are particularly simple.
To be more explicit, 
 using the coordinates induced by \eqref{descedent-coordinates} and let 
$$(\ell)_n:=\ell\cdot(\ell+1)\cdots(\ell+n-1)$$
be the Pochhammer symbol with the convention $(\ell)_0:=1$, then
the Virasoro operators $\{L_k^\E\vert k\in \mathbb{Z}; k\geq -1\}$  are given by 
\begin{align*}
L_k^{\E}=-(k+1)!{\partial\over\partial \widetilde{t}_{k+1}^0}&+\sum_{\ell\geq0}\left((\ell)_{k+1}\widetilde{t}_{\ell}^0 {\partial\over \widetilde{t}_{k+\ell}^0}+(\ell+1)_{k+1}\widetilde{t}_\ell^3{\partial\over \partial \widetilde{t}_{k+\ell}^{3}}\right)
\\&+\sum_{\ell\geq0}\left((\ell+1)_{k+1}\widetilde{t}_{\ell}^{1} {\partial\over \widetilde{t}_{k+\ell}^{1}}+(\ell)_{k+1}\widetilde{t}_\ell^2{\partial\over \partial \widetilde{t}_{k+\ell}^{2}}\right).\nonumber
\end{align*}
According to \cite[Theorem 1]{Okounkov-vir}, the total descendent GW potential defined in \eqref{descedent-curve} is annihilated by these Virasoro operators 
$$L_k^{\E} \cD^\E(\tilde{\bf t})=0\,.$$

Recently in \cite{He:2020}, 
using Givental's quantization formula of quadratic Hamiltonians \cite{Giv01}, 
the second author and his collaborator study Virasoro operators in FJRW theory and conjecture that the total ancestor FJRW potential of any admissible LG pair $(W,G)$ is annihilated by the defining Virasoro operators. Besides various generically semisimple cases, they also verified the conjecture for the non-semisimple Fermat cubic pair $(W_3, \mu_3)$, using Theorem \ref{main-theorem}. 
More explicitly, using the coordinates induced by \eqref{fjrw-coordinates}, the Virasoro operators $\{L_k^{W_3, \mu_3}\vert k\in \mathbb{Z}; k\geq -1\}$ for the Fermat cubic pair $(W_3, \mu_3)$ are 
\begin{align*}
	L_k^{W_3, \mu_3}:=-(k+1)!{\partial\over \partial t_{k}^{0}}
	&+\sum_{\ell\geq 0}^\infty\left((\ell)_{k+1}s^0_\ell\frac{\partial}{\partial s^0_{k+\ell}}+(\ell+1)_{k+1}s^3_\ell\frac{\partial}{\partial s^3_{k+\ell}}\right)\label{cubic-virasoro}\\
	&+\sum_{\ell\geq 0}^\infty\left((\ell+1)_{k+1}s^1_\ell\frac{\partial}{\partial s^1_{k+\ell}}+(\ell)_{k+1}s^2_\ell\frac{\partial}{\partial s^2_{k+\ell}}\right).\nonumber
\end{align*}
It is not hard to see that these operators commute with the quantization operator $\widehat{S}_{t}^{-1}$ in the ancestor/descendent correspondence formula \eqref{ancestor-descendent} and the holomorphic Cayley transformation $\mathscr{C}_{ \tau_* }^{\rm hol}$ in Theorem \ref{main-theorem}. 
Therefore, Virasoro constraints for the FJRW theory is a consequence of Theorem \ref{main-theorem}.
\begin{cor}
\label{virasoro-lg}
\cite{He:2020}
The total ancestor FJRW potential of the pair $(W_3, \mu_3)$ is annihilated by the Virasoro operators $\{L_k^{W_3, \mu_3}\}$,
$$L_k^{W, \mu_3} \mathcal{A}^{W_3, \mu_3}({\bf s})=0\,.$$
\end{cor}

\appendix

\section{}

\subsection{A genus-one formula for Fermat cubic polynomial}
\label{appendixIfunction}

For the examples studied in this paper,
the connection between modular forms and periods of families of elliptic curve
give rise to nice formulae for the
 holomorphic Cayley transformation of quasi-modular forms in terms of hypergeometric series and Givental's $I$-functions.
In the following, we shall only consider the $d=3$ case
as an example, the other cases are similar.

Let us first recall some facts of quasi-modular forms following the exposition in
 \cite{Shen:2017}.
Let $\Gamma(3)$ be the level-$3$ principal congruence subgroup of $\Gamma={\rm SL}(2,\mathbb{Z})/\{\pm1\}$.
It is well known that the ring of quasi-modular forms (with a certain Dirichlet character) for
$\Gamma(3)$ is generated by
$$A=\theta_{A_{2}}(2\tau)$$ and
\begin{equation}\label{eqnEE2}
E= {3E_{2}(3\tau)+ E_{2}(\tau)\over 4}\,,
\end{equation}
where $\theta_{A_{2}}$
is the theta function for the $A_{2}$-lattice.
Define further the quantities (where $\eta$ is the Dedekind eta function)
\begin{equation}\label{eqnzA}
C=3 { \eta(3\tau)^3 \over \eta(\tau)}\,,
\quad
\alpha={C^3\over A^3}\,.
\end{equation}

These quantities satisfy
\begin{equation}
A=\,_{2}F_{1}( {1\over 3}, {2\over 3}; 1;\alpha)\,,
\end{equation}
and furthermore
\begin{equation}\label{eqnA2E2}
\begin{dcases}
A^2={1\over 2} (3E_{2}(3\tau)- E_{2}(\tau))={1\over 2\pi \sqrt{-1}}{1\over \alpha(1-\alpha)}{\partial\over \partial \tau} \alpha,\quad\\
E={6\over 2\pi \sqrt{-1} } {\partial \over \partial \tau}\log A-{2C^3-A^3\over A}\,.
\end{dcases}
\end{equation}
Using \eqref{eqnEE2}, \eqref{eqnzA} and \eqref{eqnA2E2}, we can rewrite the quasi-modular form
$E_{2}$ as
\begin{align}
E_{2}(\tau)&={12\over 2\pi \sqrt{-1} } {\partial \over \partial \tau}\log A
-(4\alpha-1)A^2\label{e2-formula}\\
&={1\over 2\pi \sqrt{-1} } {\partial \over \partial \tau}\left(12\log A
+ \log  (\alpha(1-\alpha)^3)\right).\nonumber 
\end{align}
In \cite{Shen:2016} the following was obtained from period calculation.
Taking $\tau_{*}=1/(1-\zeta_{3})$ 
as given in
\eqref{eqnellipticpoint} and $c$ as in  \eqref{eqnchoiceforc},
then one has
\begin{align*}
s(\tau)&=2\pi \sqrt{-1} c (\tau_{*}-\bar{\tau}_{*}) {\tau-\tau_{*}\over \tau-\bar{\tau}_{*}}\\
&=-2\pi \sqrt{-1} c (\tau_{*}-\bar{\tau}_{*})  {\Gamma(-{1\over 3})\Gamma({2\over 3})^2 \over
\Gamma({1\over 3})^3 } (-\alpha)^{-{1\over 3}} {\,_{2}F_{1}({2\over 3}, {2\over 3};{4\over 3};\alpha^{-1}) \over
\,_{2}F_{1}({1\over 3}, {1\over 3};{2\over 3};\alpha^{-1})}\,.
\end{align*}
Furthermore, one has
\begin{align*}
\mathscr{C}_{\tau_{*}}(A)&=(2\pi \sqrt{-1} c)^{-{1\over 2}}   {\Gamma({1\over 3})\over \Gamma({2\over 3})^2}
(-\alpha)^{-{1\over 3}}\,_{2}F_{1}\left({1\over 3}, {1\over 3};{2\over 3};\alpha^{-1}\right)\,,\\
\mathscr{C}_{\tau_{*}}(C)&=(2\pi \sqrt{-1} c)^{-{1\over 2}}   {\Gamma({1\over 3})\over \Gamma({2\over 3})^2}
(-1)^{-{1\over 3}}\,_{2}F_{1}\left({1\over 3}, {1\over 3};{2\over 3};\alpha^{-1}\right)\,.
\end{align*}
Combining the properties of the holomorphic Cayley transformation, Theorem \ref{mainthm2},
and \eqref{e2-formula}, we
immediately get 
\begin{align*}
\LL \phi \RR_{1,1}^{W_3}&=\mathscr{C}_{ \tau_{*}}^{\mathrm{hol}}(\LL \omega \RR_{1,1}^{\E})\\
&=c^{-1}{\partial \over \partial s}\left(-{1\over 2}\log
\,_{2}F_{1}\left({1\over 3}, {1\over 3};{2\over 3};\alpha^{-1}\right)-{1\over 8}\log (1-\alpha^{-1})
\right)\,.
\end{align*}


In the above GW generating series, the divisor class $\omega$
which corresponds to the first Chern class of a degree one line bundle on $\E$ is used as the insertion.
According to the divisor axiom, it follows that 
\begin{equation*}
\LL \, \RR_{1,0}^{\E}=-\log \eta(\tau)\,,
\end{equation*}
up to an additive constant.
Results derived for a plane cubic curve $\E_3$, such as those in Givental's formalism, use the pull-back of the hyperplane class
on the ambient space $\mathbb{P}^2$ as the insertion. The corresponding class $H$ is related to the one $\omega$ above by $H=3\omega$.
Hence we have up to an additive constant
\begin{equation*}
\LL \, \RR_{1,0}^{\E_3}=-\log \eta(3\tau)\,,
\end{equation*}
and 
thus
\begin{equation*}
\LL H \RR_{1,0}^{\E_3}=-{1\over 24}\cdot 3\cdot E_{2}(3\tau)\,.
\end{equation*}
Using \eqref{eqnEE2}, \eqref{eqnzA} and \eqref{eqnA2E2}, one can rewrite it
as
\begin{equation*}
\LL H \RR_{1,0}^{\E_3}=
{1\over 2\pi \sqrt{-1} } {\partial \over \partial \tau}\left( -{1\over 2}\log \,_{2}F_{1}({1\over 3}, {2\over 3};1;\alpha)
-{1\over 24} \log  (\alpha^3(1-\alpha))\right)\,.
\end{equation*}
This matches the results in \cite{Zinger:2009, Popa:2013} obtained using virtual localization.
Its holomorphic Cayley transformation is
\begin{equation*}
\mathscr{C}_{ \tau_{*}}^{\mathrm{hol}}(\LL H \RR_{1,0}^{\E_3})
=c^{-1}{\partial \over \partial s}\left(-{1\over 2}\log
\,_{2}F_{1}\left({1\over 3}, {1\over 3};{2\over 3};\alpha^{-1}\right)-{1\over 24}\log (1-\alpha^{-1})
\right)\,.
\end{equation*}
This agrees with the result derived using the wall-crossing method in Guo-Ross \cite{Guo:2016}.

\subsection{Cayley transformation and $I$-functions }
Now we discussion the connection between our formulation of LG/CY correspondence and the original formulation in \cite[Conjecture 3.2.1]{Chiodo:2010} using $I$-functions.

\subsubsection{$I$-functions and analytic continuation}
Following \cite[Section 4.2]{Chiodo:2010},  the cohomology-valued Givental $I$-function for the GW theory of the cubic hypersurface 
$$\{W_3=x_1^3+x_2^3+x_3^3=0\}\subset \mathbb{P}^2$$ 
is given by\footnote{Here the variable $\mathrm{q}$ should not be confused with the variable
$q=e^{2\pi i \tau}$ in modular forms.
}
\begin{equation}
\begin{aligned}
I_{\rm GW}(\mathrm{q}, z)&:=\sum_{d\geq 0} z\,\mathrm{q}^{H/ z+d}{\prod_{k=1}^{3d}(3H+kz)\over \prod_{k=1}^{d}(H+k z)^3}\\
&=I_0^{\rm GW}(\mathrm{q})\,z\, \one +I_1^{\rm GW}(\mathrm{q})H\,,
\end{aligned}
\end{equation}
where $H$ is the hyperplane class of $\mathbb{P}^{2}$.
While the $I$-function for the FJRW theory of the pair $(W_3, \mu_3)$ is given by
\begin{equation}
\begin{aligned}
I_{\rm FJRW}(t, z)&:= z\sum_{k=1}^{2}{1\over \Gamma(k)}\cdot \sum_{\ell\geq0}
{\left(({k\over 3})_{\ell}\right)^3t^{k+3\ell}\over (k)_{3\ell} z^{k-1}}\phi_{k-1}\\
&=I_0^{\rm FJRW}(t)\, z\, \one +I_1^{\rm FJRW}(t)\phi\,,
\end{aligned}
\end{equation}
where $\phi_0=\one$ and $\phi_{1}=\phi$ are nontrivial degree-zero and two elements in the state space.
The genus-zero LG/CY correspondence \cite{Chiodo:2010} relates these two $I$-functions by analytic continuation via
 $\mathrm{q}=t^{-3}$.
To be more explicit, one has the following analytic continuation
 \begin{align}
\begin{pmatrix}
I_1^{\rm FJRW}(t)/3\\
I_0^{\rm FJRW}(t)/3
\end{pmatrix}
=\begin{pmatrix}
{(-1)\over \Gamma^3({2\over 3})}{2\pi\sqrt{-1}\zeta_3\over 1-\zeta_3} & -{(-1)\over \Gamma^3({2\over 3})}{(2\pi\sqrt{-1})^2\zeta_3\over (1-\zeta_3)^2}\\
-{(-1)^2\over \Gamma^3({1\over 3})}{2\pi\sqrt{-1}\zeta_3^2\over 1-\zeta_3^2} &{(-1)^2\over \Gamma^3({1\over 3})}{(2\pi\sqrt{-1})^2\zeta_3^2\over (1-\zeta_3^2)^2} 
\end{pmatrix}
\begin{pmatrix}
I_1^{\rm GW}(t({\mathrm{q}}))\\
I_0^{\rm GW}(t({\mathrm{q}}))
\end{pmatrix}\label{i-function-matrix}\,,
\end{align}
where the normalization factor $1/3$ on the basis $\{I_{0}^{\rm FJRW}, I_{1}^{\rm FJRW}\}$
is introduced such that the connection matrix lies in $\mathrm{SL}_{2}(\mathbb{C})$.
In particular, define
\begin{equation}\label{eqnmirrormap}
t_{\rm GW}:={I_1^{\rm GW}(\mathrm{q})\over I_0^{\rm GW}(\mathrm{q})}\,,
\quad
t_{\rm FJRW}:={I_1^{\rm FJRW}(t)\over I_0^{\rm FJRW}(t)}
\,.
\end{equation}
Then one has
\begin{equation}\label{eqnsymplectictransformation}
t_{\rm FJRW}
=-e^{\pi i\over 3}\cdot {\Gamma({1\over 3})^3\over \Gamma(-{1\over 3}) \Gamma({2\over 3})^{2}}\cdot { t_{\rm GW} -2\pi i \tau_{*}\over t_{\rm GW} -2\pi i \bar{\tau}_{*} }\,.
\end{equation}

 \subsubsection{Cayley transformation}
Following the computations in \cite{Shen:2016} as in Appendix \ref{appendixIfunction}, we can relate the above $I$-functions to modular forms.
In particular, we see that
\begin{equation}\label{eqnmirrormap}
t_{\rm GW}:={I_1^{\rm GW}(\mathrm{q})\over I_0^{\rm GW}(\mathrm{q})}=2\pi i \tau\,,
\quad
t_{\rm FJRW}:={I_1^{\rm FJRW}(t)\over I_0^{\rm FJRW}(t)}
=e^{2\pi i\over 3}\cdot -{\sqrt{3}\over i}\cdot {\Gamma({1\over 3})^2\over \Gamma(-{1\over 3})}s\,.
\end{equation}
Here $s$ is the coordinate given in \eqref{eqnCayleytransform}, with again
 $\tau_{*}=1/(1-\zeta_{3})$ 
as given in
\eqref{eqnellipticpoint} and $c$ as in  \eqref{eqnchoiceforc}.
Analytical continuations on
the $I$-functions,
 induced by 
\eqref{eqnsymplectictransformation}, coincide with Cayley transformations on them induced by \eqref{eqnCayleytransform} by construction  \cite{Shen:2016}.

Through the connection to modular forms,
LG/CY correspondence on $I$-functions can be restated as follows.
Let $\mathcal{M}=\Gamma(3)\backslash \mathbb{H}^{*}$ 
be the modular curve as the global moduli space, where $\mathbb{H}^{*}=\mathbb{H}\cup \mathbb{P}^{1}(\mathbb{Q})$. Denote its canonical bundle by $K_{\mathcal{M}}$.
Then $I^{\rm GW}$
and $I^{\rm FJRW}$
correspond to descriptions of the same holomorphic section
of the line bundle that is isomorphic to $K_{\mathcal{M}}^{\otimes {1\over 2}}$, but on different patches of the moduli space.
Their coordinate expressions $I_0^{\rm GW}, I_{0}^{\rm FJRW}$, with respect to the trivializations $(d\tau)^{1\over 2}, (ds)^{1\over 2}$ respectively,
are modular forms related by Cayley transformation.

\subsubsection{Stationary correlation functions}

At higher genus, consider the stationary correlation function
$$\LL\alpha_1\psi_1^{\ell_1},\cdots,\alpha_n\psi_n^{\ell_n}\RR_{g,n}^{\clubsuit}\,,$$
with $\alpha_i=\omega$ when $\clubsuit=\E_3$ and $\alpha_i=\phi$ when $\clubsuit=W_3$. By applying the $g$-reduction technique in Lemma \ref{g-reduction} inductively, we see that
under the map  \eqref{eqnmirrormap}
this correlation function on the GW side is the Fourier expansion of a quasi-modular form of weight $2g-2+2n$ near the cusp, and on the FJRW side is the Taylor expansion (in terms of the parameter $s$) of the same quasi-modular form
near the point $\tau_{*}$. 

According to standard facts in the theory of modular forms (see e.g., \cite{Urban:2014nearly, Zagier:2008}) on the transition between quasi-modular forms and almost-holomorphic modular forms, we see that on the level of GW correlation functions the modular completion is induced by the transformation mapping the frame of $H^{\mathrm{even}}(\E_3, \mathbb{C})$ from
$\{ \one+ 2\pi i \tau H, 2\pi i H\}$
to
$\{ \one+ 2\pi i\tau H,  {1\over \bar{\tau}-\tau}(\one-2\pi i \bar{\tau} H)\}$.
This transformation also induces  the modular completion on the FJRW correlation functions by compositing with the aforementioned transformation that relates $I^{\rm GW}$
with $I^{\rm FJRW}$.

One succinct way to reformulate our higher-genus LG/CY correspondence result 
on  $\LL\alpha_1\psi_1^{\ell_1},\cdots,\alpha_n\psi_n^{\ell_n}\RR_{g,n}^{\clubsuit}$
is then as follows.
Denote its modular completion by
$$\LL\alpha_1\psi_1^{\ell_1},\cdots,\alpha_n\psi_n^{\ell_n}\RR_{g,n}^{\clubsuit, \,\widehat{~}}\,.$$
Let $I_0^{\clubsuit}=I_{0}^{\rm GW}, d\tau^{\clubsuit}=d\tau$ for $\clubsuit=\E_3$, and
$I_0^{\clubsuit}=I_{0}^{\rm FJRW}, d\tau^{\clubsuit}=ds$ for $\clubsuit=W_3$.
Then
the quantity
$$(I_0^{\clubsuit})^{2-2g}\LL\alpha_1\psi_1^{\ell_1},\cdots,\alpha_n\psi_n^{\ell_n}\RR_{g,n}^{\clubsuit,\,\widehat{~}}(d\tau^{\clubsuit})^{\otimes n}\,$$
is a global (smooth with holomorphic pole) section 
of the holomorphic line bundle $K_{\mathcal{M}}^{\otimes n}$
on the modular curve $\mathcal{M}=\Gamma(3)\backslash \mathbb{H}^{*}$.

\bigskip{}

\noindent{\small Shanghai Center for Mathematical Sciences, Fudan University, Shanghai, China}


\noindent{\small E-mail: \tt lijun2210@fudan.edu.cn}

\medskip{}

\noindent{\small Department of Mathematics, University of Oregon, Eugene, Oregon, USA}

\noindent{\small E-mail: \tt yfshen@uoregon.edu}

\medskip{}

\noindent{\small Yau Mathematical Sciences Center, Tsinghua University, Beijing 100084, China}

\noindent{\small E-mail: \tt jzhou2018@mail.tsinghua.edu.cn}

\end{document}